\documentclass[reqno]{amsart}
 
\usepackage{amsmath}
\usepackage{amsthm,amssymb,color,comment}
\usepackage{bbm}
\usepackage{hyperref}
\usepackage[TS1,T1]{fontenc}
\usepackage[utf8]{inputenc}
\usepackage{dsfont}
\usepackage{tikz}
\usepackage{enumitem}
\usepackage{subfigure}
\usepackage{mathtools}
\usepackage{bbold}
\usepackage[sort]{cite}

\numberwithin{equation}{section}

\newtheorem{thm}{Theorem}[section]

\newtheorem{lem}[thm]{Lemma}

\newtheorem{Def}[thm]{Definition}

\theoremstyle{definition}
\newtheorem{Ass}[thm]{Assumption}
\newtheorem{rem}[thm]{Remark}

\newtheorem{Nott}[thm]{Notation}
\newtheorem{Ex}[thm]{Example}

\DeclareMathOperator{\DIV}{div}

\newcommand{\R}{\mathbb{R}}
\newcommand{\Td}{\mathbb{T}^{d}}
\newcommand{\N}{\mathbb{N}}
\newcommand{\p}{\partial}
\newcommand {\f}{\frac}
\newcommand{\eps}{\varepsilon}

\newcommand{\diff}{\mathop{}\!\mathrm{d}}
\newcommand{\ie}{\emph{i.e.}\;}

\newcommand{\ueps}{u_{\varepsilon}}

\newcommand{\weak}{\rightharpoonup}
\DeclarePairedDelimiter{\norm}{\lVert}{\rVert}

\makeatletter
\newcommand{\doublewidetilde}[1]{{%
  \mathpalette\double@widetilde{#1}%
}}
\newcommand{\double@widetilde}[2]{%
  \sbox\z@{$\m@th#1\widetilde{#2}$}%
  \ht\z@=.9\ht\z@
  \widetilde{\box\z@}%
}
\makeatother
\author{Charles Elbar}
\address{{\it Charles Elbar:} Sorbonne Universit\'{e}, Laboratoire Jacques-Louis Lions (LJLL), F-75005 Paris, France}
\email{charles.elbar@sorbonne-universite.fr}
\thanks{}

\author{Jakub Skrzeczkowski}
\address{{\it Jakub Skrzeczkowski: } Institute of Mathematics of Polish Academy of Sciences; Faculty of Mathematics, Informatics and Mechanics, University of Warsaw, Poland}
\email{jakub.skrzeczkowski@student.uw.edu.pl}
\thanks{J.S. was supported by the National Agency of Academic Exchange project "Singular limits in parabolic equations" no. BPN/BEK/2021/1/00044. Authors are grateful to Benoît Perthame for fruitful discussions and helpful suggestions that greatly improved this paper.	}

\begin{document}

\title[]{Degenerate Cahn-Hilliard equation: From nonlocal to local}

\begin{abstract}
There has been recently an important interest in deriving rigorously the Cahn-Hilliard equation from the nonlocal equation, also called aggregation equation. So far, only non-degenerate mobilities were treated. 
Since we are motivated by models for the biomechanics of living tissues, it is useful to include degenerate motilities. In this framework, we present a new method to show the convergence of the nonlocal to the local {\em degenerate} Cahn-Hilliard equation. The method includes the use of nonlocal Poincaré and compactness  inequalities. 
\end{abstract}



\maketitle
\setcounter{tocdepth}{1}

{\bf Conflict of interest statement:} The authors have no conflicts of interest to declare that are relevant to the content of this article.\\

{\bf Data availability statement:} Data availability is not applicable to this article as no new data were created or analyzed in this study.

\section{Introduction}
Several recent papers~\cite{MR4408204,MR4093616,MR4198717,MR4248454} have addressed the problem of  deriving rigorously the Cahn-Hilliard equation from the nonlocal equation, also called aggregation equation~\cite{carrillo2019aggregation}. In these works, only the case of non-degenerate mobilities is treated, which avoids the delicate question of defining the limit of low-order products that one encounters for  nonlocal degenerate mobility that we present now. The degenerate model is written
 \begin{align*}
\partial_t u = \DIV(u \nabla \mu),\quad \text{in}\quad &(0,+\infty)\times \Td,\\
\mu = B[u] + F'(u),\quad \text{in}\quad &(0,+\infty)\times \Td, 
\end{align*}
equipped with an initial datum $u^{0}\ge 0$. Here, $\Td$ is the $d$-dimensional flat torus, $B$ is the nonlocal operator $B=B_{\varepsilon}$ defined with
\begin{equation}\label{operatorB}
B_\eps[u](x) = \f{1}{\eps^{2}}(u(x)-\omega_{\eps}\ast u(x))=\f{1}{\eps^{2}}\int_{\Td}\omega_{\eps}(y)(u(x)-u(x-y)) \diff y
\end{equation}
for $\eps$ small enough and $\omega_{\eps}$ is the usual radial mollification kernel $\omega_{\eps}(x)=\frac{1}{\eps^{d}}\omega(\frac{x}{\eps})$ with $\omega$ compactly supported in the unit ball of $\R^{d}$ satisfying
\begin{equation}
\int_{\R^{d}}\omega(y) \diff y =1, \quad \int_{\R^{d}} y\, \omega(y) \diff y =0,  \quad \int_{\R^{d}}   y_i y_j \omega \diff y = \delta_{i,j}\f{2D}{d} , \quad \int_{\R^{d}} \omega(y) |y|^3 \diff y < \infty,
\label{as:omega}
\end{equation}

\noindent for some constant $D>0$. Our target is to prove that as $\varepsilon \to 0$, the constructed solutions of 
 \begin{align}
\partial_t u = \DIV(u \nabla \mu),\quad \text{in}\quad &(0,+\infty)\times \Td,\label{eq:CHE1}\\
\mu = B_{\eps}[u] + F'(u),\quad \text{in}\quad &(0,+\infty)\times \Td \label{eq:CHE2},
\end{align}
tend to the weak solution of the degenerate Cahn-Hilliard equation
 \begin{align}
\partial_t u = \DIV(u \nabla \mu),\quad \text{in}\quad &(0,+\infty)\times \Td,\label{eq:CH1}\\
\mu = -D\Delta u + F'(u),\quad \text{in}\quad &(0,+\infty)\times \Td \label{eq:CH2}.
\end{align}

The product $u \nabla \mu$ in the limiting system is not a priori well defined since we cannot control third-order derivatives. Passing to the limit is thus the challenge we overcome here.

Our motivation for this work is fourfold. 

\begin{itemize}
    \item Firstly, the interest for the nonlocal Cahn-Hilliard equation is an old problem that can be traced back to Giacomin and Lebowitz~\cite{MR1453735,MR1638739}. These seminal works establish the derivation of the degenerate nonlocal Cahn-Hilliard equation departing from stochastic systems of particles. However, they left open the question of deriving the local degenerate  Cahn-Hilliard equation from the nonlocal one. This is the challenge we overcome here.
    \item Secondly, a revival of interest for this problem appeared in the last years with several papers~\cite{MR4408204,MR4093616,MR4198717,MR4248454} deriving the local from the nonlocal Cahn-Hilliard equation in the non-degenerate case.
    \item Third, the nonlocal Cahn-Hilliard equation can be seen as a porous medium equation with a smooth advection term that is well understood, conversely to the local degenerate Cahn-Hilliard equation.   \item Finally, the nonlocal Cahn-Hilliard equation \eqref{eq:CHE1}--\eqref{eq:CHE2} is in fact an aggregation-diffusion equation with a nonlocal term corresponding to the aggregation effect \cite{carrillo2019aggregation}. Thus, in this paper, we show that if the nonlocal effect is appropriately scaled, one approaches Cahn-Hilliard equation. This limit was formally stated for instance in \cite{delgadino2018convergence,bernoff2016biological,falco2022local} and our work provides a rigorous mathematical argument for this approximation.
\end{itemize}

\subsection{Main result}

We make the following assumptions on the potential $F$.

\begin{Ass}[potential $F$]\label{ass:potentialF} For the interaction potential we assume that there exists $k \geq 2$ and a decomposition $F = F_1 + F_2$ such that
\begin{enumerate}[label=(\Alph*)]
    \item \label{assumption_pot_1} $F_1$, $F_2$ are of class $C^2$,
    \item \label{assumption_pot_3} $F_1 = 0$ or $F_1$ is a convex function which has $k$-growth in the sense that for some nonnegative constants $C_{1}$, ..., $C_8$ we have
$$
C_{1}|u|^k - C_{2}  \leq F_1(u) \leq C_{3}|u|^k + C_{4}.  
$$
$$
 C_{5}|u|^{(k-2)} - C_{6}\leq F_1''(u) \leq C_{7}|u|^{(k-2)} + C_{8}, 
$$
   \item \label{assumption_pot_4} $F_2$ has bounded second derivative \ie $\|F_2''\|_{\infty} < \infty$ and $F_2(u) \geq -C_{9} - C_{10}\, u^2$ where $C_{10}$ is sufficiently small: more precisely $4\,C_{10} < C_p$ with $C_p$ being the constant in Lemma~\ref{lem:Poincare_with_average}.
\end{enumerate}
\end{Ass}

\begin{Ex}\label{ex:potentials}
The following potentials satisfy Assumption \ref{ass:potentialF}.
\begin{enumerate}[label=(\arabic*)]    \item power-type potential $F(u) = |u|^{\gamma}$, $\gamma>2$ used in the context of tumour growth models \cite{Perthame-Hele-Shaw,david2021incompressible,elbar2021degenerate,DEBIEC2021204},
    \item double-well potential $F(u) = u^2\,(u-1)^2$ which is an approximation of logarithmic double-well potential often used in Cahn-Hilliard equation, see \cite[Chapter 1]{MR4001523},
    \item\label{ex_general_F} any $F \in C^2$ such that for some interval $I \subset \R$ we have $F''(u)>a>0$ for $u \in \R \setminus I$ and 
\begin{align*}
C\, |u|^k - C &\leq F(u) \leq C \, |u|^k + C &\mbox{ for all } u \in \R \setminus I,\\
C\, |u|^{k-2} - C &\leq F''(u) \leq C \, |u|^{k-2} + C  &\mbox{ for all } u \in \R \setminus I,
\end{align*}
see Lemma \ref{lem:potentials_satisfying_ass} for details.
\end{enumerate}
Note that (3) is a more general version of (2).
\end{Ex}

\begin{Nott}[exponents $s$ and $k$] \label{not:s_and_k}
In what follows we write
$$
k = 
\begin{cases}
2 &\mbox{ if } F_1=0,\\
k &\mbox{ if } F_1\neq0.\\
\end{cases}
$$
We also define $s = \frac{2k}{k-1}$ and $s'$ its conjugate exponent.
\end{Nott}

Now, we define weak solutions of the nonlocal and local degenerate Cahn-Hilliard equation.

\begin{Def}
\label{def:weak_sol_local}\noindent We say that $\ueps$ is a weak solution of \eqref{eq:CHE1}-\eqref{eq:CHE2} if
\begin{align*}
&\ueps \in L^{\infty}(0,T;L^{k}(\Td)), \qquad \quad \p_{t} \ueps \in L^{2}(0,T;W^{-1,s'}(\Td)) \\
&\nabla \ueps \in L^2((0,T)\times\Td), \qquad \quad
\sqrt{F_{1}''(\ueps)}\nabla \ueps\in L^2((0,T)\times\Td),
\end{align*}
$u(0,x)=u_0(x)$ a.e. in $\Td$ and for all $\varphi \in L^2(0,T; W^{1,\infty}(\Td))$
\begin{equation}\label{eq:weaksol}
    \int_{0}^{T}\langle\p_{t} \ueps,\varphi\rangle_{(W^{-1,s'}(\Td), W^{1,s}(\Td))}=-\int_{0}^{T}\int_{\Td} \ueps \nabla B_{\varepsilon}[u_{\varepsilon}]\cdot\nabla\varphi-\int_{0}^{T}\int_{\Td}\ueps F''(\ueps)\nabla \ueps \cdot\nabla\varphi.
\end{equation}
\end{Def}

\begin{Def}\label{def:weak_sol_limit}
We say that $u$ is a weak solution of~\eqref{eq:CH1}-\eqref{eq:CH2} if 
\begin{align*}
&u \in L^{\infty}(0,T;L^{k}(\Td))\cap L^{2}(0,T;H^{2}(\Td)),  \quad \p_{t} u \in L^{2}(0,T;W^{-1,s'}(\Td)),  \\
&\sqrt{F_{1}''(u)}\nabla u\in L^2((0,T)\times\Td),
\end{align*}
 $u(0,x)=u_0(x)$ a.e. in $\Td$ and if for all $\varphi \in L^2(0,T; W^{2,\infty}(\Td))$ we have 
\begin{multline*}
\int_{0}^{T}\langle\p_{t}u,\varphi\rangle_{(W^{-1,s'}(\Td),W^{1,s}(\Td))} = -D\int_0^T \int_{\Td}\Delta u\, \nabla u \cdot \nabla\varphi 
-D\int_0^T \int_{\Td}u\,\Delta u\,\Delta\varphi \\-\int_{0}^{T}\int_{\Td}u\,F''(u)\,\nabla u\cdot\nabla\varphi.     
\end{multline*}
\end{Def}
\begin{rem}[initial condition]
In Definitions \ref{def:weak_sol_local} and \ref{def:weak_sol_limit} we can evaluate pointwise value $u(0,x)$ because by \cite[Lemma 7.1]{MR3014456}, we know that $u \in C(0,T; W^{-1,s'}(\Td))$.
\end{rem}

\noindent With these assumptions we can construct solutions to~\eqref{eq:CHE1}-\eqref{eq:CHE2}.
\begin{thm}[Existence of solutions for the nonlocal system]\label{thm:weaksoldelta}
Let $\varepsilon_0$ be given by 
\begin{equation}\label{eq:eps_0_restriction}
\varepsilon_0 := \min\left(\varepsilon_0^A, \varepsilon_0^B, \frac{1}{\sqrt{2\, \|F_2''\|_{\infty}}}\right)    
\end{equation}
where $\varepsilon_0^A$ and $\varepsilon_0^B$ are given in Lemma \ref{lem:Poincare_with_average} and \ref{lem:poincare_nonlocal_H1_L2} respectively. Let $\varepsilon < \varepsilon_0$. Let $u^{0}\ge 0$ be an initial datum with finite energy and entropy $E_{\eps}(u^0), \Phi(u^0) < \infty$ defined in ~\eqref{eq:intro_energy}-\eqref{eq:intro_entropy}. There exists a global weak solution $\ueps$ of~\eqref{eq:CHE1}-\eqref{eq:CHE2} in the sense defined by Definition~\ref{def:weak_sol_local}. It satisfies the dissipation of energy and entropy \eqref{eq:energy1}-\eqref{eq:entropy1} with $u = \ueps$, $\mu = \mu_{\varepsilon}$. Moreover, $\ueps \geq 0$.
\end{thm}

\noindent Our main result reads as follows.

\begin{thm}[Convergence of nonlocal to local Cahn-Hilliard equation on the torus]\label{thm:final}
Let $u^{0}\ge 0$ be an initial datum with finite energy and entropy $E(u^0), \Phi(u^0) < \infty$ defined in ~\eqref{eq:final_energy} and \eqref{eq:intro_entropy}. Let $\{u_{\varepsilon}\}$ be a sequence of solutions of the degenerate nonlocal Cahn-Hilliard equation~\eqref{eq:CHE1}-\eqref{eq:CHE2} from Theorem \ref{thm:weaksoldelta}. Then, up to a subsequence,
$$
u_{\varepsilon} \to u \mbox{ in } L^2(0,T; H^1(\Td))
$$
where $u$ is a weak solution of the degenerate Cahn-Hilliard equation~\eqref{eq:CH1}-\eqref{eq:CH2} as in Definition~\ref{def:weak_sol_limit}.   
\end{thm}
\begin{rem}
Note that by Lemma \ref{lem:inv_poincare_ineq} condition $E(u^0)<\infty$ implies that $E_{\eps}(u^0)<\infty$.
\end{rem}


\subsection{Important components of the proof.} There are three main ingredients of the proof. 
\begin{itemize}
    \item Compactness for the system \ref{eq:CHE1}--\ref{eq:CHE2} is obtained from the energy $E_{\varepsilon}$ and entropy $\Phi$
     \begin{align}
     &E_{\varepsilon}[u]:=\int_{\Td}F(u)\diff x+\f{1}{4\eps^{2}}\int_{\Td}\int_{\Td}\omega_{\eps}(y)|u(x)-u(x-y)|^{2}\diff x\diff y, \label{eq:intro_energy}
     \\ &\Phi[u]:=\int_{\Td}u(\log(u)-1)+1 \diff x \label{eq:intro_entropy}
     \end{align}
     Their dissipation is formally controlled by the identities
\begin{align}
&E_{\varepsilon}[u](t) + \int_{0}^{t}\int_{\Td} u \, |\nabla \mu_{\eps}|^2 \le E_{\varepsilon}[u^{0}],\label{eq:energy1}\\
&\Phi[u](t) + \frac{1}{2\eps^{2}}\int_{0}^{t}\int_{\Td} \int_{\Td} \omega_{\varepsilon}(y) \, |\nabla u(x) - \nabla u(x-y)|^2 + \int_{\Td} F''(u) |\nabla u|^{2} \le \Phi[u^{0}].\label{eq:entropy1}
\end{align}
     According to the result of Bourgain-Brézis-Mironescu~\cite{bourgain2001another} which was improved later by Ponce~\cite{MR2041005}, uniform bounds from \eqref{eq:energy1}, \eqref{eq:entropy1} together with Lions-Aubin lemma, yields strong convergence of $\{u_{\varepsilon}\}$ and $\{\nabla u_{\varepsilon}\}$ to $u,\nabla u$ in $L^2((0,T)\times \Td)$. We note that in the limit $\varepsilon \to 0$, the energy $E_{\varepsilon}[u_{\eps}]$ satisfy (see~\cite[Theorem 4]{bourgain2001another} and~\cite[Theorem 1.2]{MR2041005}) 
     \begin{equation}\label{eq:final_energy}
     E[u] = \int_{\Td}F(u)\diff x+\f{C(d)}{2} \int_{\Td} |\nabla u(x)|^{2}\diff x\le \liminf_{\eps\to 0}E_{\eps}[u_{\eps}] 
     \end{equation}
 for some constant $C(d)$ depending only on the dimension $d$. Similarly, for the nonlocal term in the dissipation of the entropy we have 
	$$
	C(d)\,\sum_{i,j=1}^d \int_0^t \int_{\Td} |\partial_{x_i} \partial_{x_j} u|^2\le \liminf_{\eps\to 0}	\frac{1}{2\eps^{2}}\int_{0}^{t}\int_{\Td} \int_{\Td} \omega_{\varepsilon}(y) \, |\nabla u_\eps(x) - \nabla u_\eps(x-y)|^2  
	$$
so in the limit $\varepsilon \to 0$ we gain one more derivative. We also point out that one can prove rigorously that \eqref{eq:CH1}--\eqref{eq:CH2} is a gradient flow of \eqref{eq:final_energy} \cite{LISINI2012814,MR2581977}.
         \item In passing to the limit, we exploit the appropriate definition of weak solutions to \eqref{eq:CH1}--\eqref{eq:CH2}. Indeed, first we prove convergence to the formulation
         \begin{align*}
&\int_{0}^{T}\langle\p_{t}u,\varphi\rangle_{(W^{-1,s'}(\Td),W^{1,s}(\Td))} = D\int_0^T \int_{\Td}(\nabla u\otimes \nabla u) : D^2 \varphi \,+
\\
& \qquad \qquad +  \f{D}{2}\int_{0}^{T}\int_{\Td}|\nabla u|^{2}\Delta\varphi
+D\int_0^T \int_{\Td}u\nabla u\cdot \nabla\Delta\varphi -\int_{0}^{T}\int_{\Td}uF''(u)\nabla u\cdot\nabla\varphi.   
\end{align*}
         Formally, it is obtained by integrating by parts twice using the formula
\begin{equation}\label{integration_by_parts}
\nabla u\Delta u=\DIV(\nabla u\otimes \nabla u)-\f{1}{2}\nabla |\nabla u|^{2}.
\end{equation}
Its main advantage is that it exploits at most first-order derivatives so that we do not need any estimates on the second-order derivatives. This is important as they are not available for nonlocal degenerate Cahn-Hilliard. More precisely, the main difficulty is non-degeneracy of \eqref{eq:CHE1}--\eqref{eq:CHE2}, that is we loose estimates on $\nabla \mu_{\eps}$ whenever $u_{\eps}$ is approaching the zone $\{u_{\eps} = 0\}$. For the non-degenerate equation studied in \cite{MR4408204,MR4093616,MR4198717,MR4248454},
     \begin{align}
\partial_t u_{\varepsilon} = \DIV \nabla \mu_{\eps} ,\quad \text{in}\quad &(0,+\infty)\times \Td,\label{eq:CHE1_nondeg}\\
\mu_{\eps} = B_{\eps}[u_{\eps}] + F'(u_\eps),\quad \text{in}\quad &(0,+\infty)\times \Td \label{eq:CHE2_nondeg},
\end{align}
one obtains immediately an estimate on $\nabla \mu_{\varepsilon}$ (by multiplying by $\mu_{\eps}$) and then one can identify its limit. Nevertheless, we point out that in \cite{MR4408204,MR4093616,MR4198717,MR4248454} the difficulty is rather the regularity of the potential and the kernel which we do not address in our work, assuming that $F$ and $\omega$ are sufficiently smooth.
\item For the nonlocal Laplacian operator given by $B_{\varepsilon}$ defined in \eqref{operatorB}, we find an operator $S_{\varepsilon}$ given in \eqref{eq:nonlocal_gradient} which resembles gradient operator. It satisfies the integration by parts formula~\ref{propS_nonneg} in Lemma~\ref{lem:S_properties} as well as the product rule~\ref{propS_product_rule} in Lemma \ref{lem:S_properties} with an error that vanishes when $\varepsilon \to 0$. This is necessary to perform usual calculus operations before sending $\varepsilon \to 0$, that is when we do not have Laplace operator in the equation.
\end{itemize} 

\subsection{Literature review and relevancy of the system}

\paragraph{\underline{The Cahn Hilliard equation}}
The equation represents a mathematical model which is widely used to describe phase transitions in fluids and living tissues.
In biology, the equation can model the morphological evolution of a growing solid tumor. There are many factors that come into play when approaching a tumor growth model, e.g. cell-cell and cell-matrix adhesion, as well as cell motility and mechanical stress. Nevertheless, the mathematical study and implementation of a model may provide very interesting information on tumor progression. Let us consider a bounded, open tissue domain, in which a tumor is evolving. One possible approach is to describe the tumoral and healthy tissues with volume fractions. By remarking that internal adhesive forces tend to bind the tumor cells together, we should take into account that phase separation may occur between healthy and tumoral tissue domains. As a result a boundary layer of finite thickness $\sqrt{D}$, where $D$ is defined below, may form between them. 

Being of fourth-order, the (local) Cahn-Hilliard equation is often rewritten in a system of two second-order equations, \ie
\begin{equation}
    \p_t u = \DIV\left(m(u)\nabla \left(F^\prime(u)-D \Delta u \right)\right) \to \begin{cases}
        \p_t u &= \DIV\left(m(u)\nabla \mu \right),\\
        \mu &= -D \Delta u +F^\prime(u),
    \end{cases}
    \label{eq:Cahn-Hilliard}
\end{equation}
where $u$ is the concentration of a phase and $\mu$ is called the chemical potential in material sciences but is often used as an effective pressure for living tissues \cite{degond2022multi,elbar2021degenerate}. Also, the interaction potential $F(u)$ contained in this effective pressure term comprises the effects of attraction and repulsion between cells. The physically relevant form of this potential is a double-well logarithmic potential and is often approximated by a smooth polynomial function. However, recent studies show that for the modeling of living tissues and for the particular application where only one of the components of the mixture experiences attractive and repulsive forces, a single-well logarithmic potential is more relevant \cite{byrne_modelling_2004}.

The existence and uniqueness of solutions for the Cahn-Hilliard system \eqref{eq:Cahn-Hilliard} strictly depends on the properties of the mobility term $m(u)$ and the potential $F(u)$, as well as the conditions assigned on the boundary. More specifically, the presence of \textit{degeneration} on the mobility, i.e. the possibility for it to vanish, can turn the analysis of solutions into a rather complex problem.

\paragraph{\underline{From nonlocal to local Cahn-Hilliard}}

The nonlocal Cahn-Hilliard equation was first obtained by Giacomin and Lebowitz \cite{MR1453735,MR1638739} by starting from a microscopic description. The model is a
$d$-dimensional lattice gas evolving via Kawasaki exchange dynamics, which is a Poisson nearest neighbor exchange process. In the hydrodynamic limit, they find that the  empirical average of the occupation numbers over a small macroscopic volume element tends to a solution of a nonlocal Cahn-Hilliard equation. This latter equation is an approximation of the local Cahn-Hilliard equation, as shown in Theorem~\ref{thm:final}. Let us also remark that there are possibly different variants of non-local Cahn-Hilliard equation, see for instance ~\cite{elbar-mason-perthame-skrzeczkowski} where a version of nonlocal Cahn-Hilliard equation is derived starting from a kinetic description inspired by~\cite{takata2018simple}.

The literature concerning the nonlocal Cahn-Hilliard equation is quite well developed and we refer for instance to \cite{MR3688414,MR3072989,MR1612250,KNOPF2021236,MR4365199,MR4221297} and \cite{MR4241616} for the cases of non-degenerate and degenerate mobilities respectively. On the other hand, for the passage to the limit of the nonlocal Cahn-Hilliard equation to the local Cahn-Hilliard equation, the existing results \cite{MR4408204,MR4093616,MR4198717,MR4248454,MR3362777} cover only the case of constant mobility. In~\cite{MR4408204} the authors prove the convergence on the torus. In \cite{MR4093616} a wide class of potentials is considered and the study is made on the torus. In~\cite{MR4198717} the convergence is obtained in the case of a bounded domain with Neumann boundary conditions and a viscosity term. Finally, in~\cite{MR4248454}, the limit is achieved with a $W^{1,1}$ kernel and a wide class of singular potentials, with Neumann boundary conditions.

\subsection{Open problem concerning bounded domains}\label{subsect:intro_bounded_domain}

One can ask if the same results hold when $\Td$ is replaced with some general bounded domain $\Omega$. More precisely, we focus on the system
\begin{align}
\partial_t u_{\varepsilon} = \DIV(u_{\varepsilon} \nabla \mu_{\eps}),\quad \text{in}\quad &(0,+\infty)\times \Omega,\label{eq:CHE1_2}\\
\mu_{\eps} = B_{\eps}[u_{\eps}] + F'(u_\eps),\quad \text{in}\quad &(0,+\infty)\times \Omega \label{eq:CHE2_2}.
\end{align}
Defining $\vec{n}$ the outward normal vector to $\p\Omega$ we impose the Neumann boundary condition
\begin{equation}\label{eq:Neumann}
u_{\eps}\f{\p \mu_{\eps}}{\p\vec{n}}=0\quad \text{on $\p\Omega$}.    
\end{equation}
The operator $B_{\eps}$ satisfies
\begin{equation}\label{operatorB_2}
B_\eps[u_{\eps}](x) = \f{1}{\eps^{2}}\left(\int_{\Omega}\omega_{\eps}(x-y)\diff y\, u_{\eps}(x)-\omega_{\eps}\ast u_{\eps}(x)\right)=\f{1}{\eps^{2}}\int_{\Omega}\omega_{\eps}(x-y)(u_{\eps}(x)-u_{\eps}(y)) \diff y.
\end{equation}
Notice that in the case $\Omega=\Td$,  this definition is the same than~\eqref{operatorB} up to a change of variable in the integral. However, since $u_{\eps}$ is not a priori defined outside $\Omega$ we need to put the argument $(x-y)$ on $\omega_{\eps}$.\\

In the limit, we expect to obtain solutions to 
 \begin{align}
\partial_t u = \DIV(u \nabla \mu),\quad \text{in}\quad &(0,+\infty)\times \Omega,\label{eq:CH1_bdd_dom}\\
\mu = -D\Delta u + F'(u),\quad \text{in}\quad &(0,+\infty)\times \Omega \label{eq:CH2_bdd_dom}\\
\f{\p u}{\p\vec{n}} = u\,\f{\p \mu}{\p\vec{n}}=0, \quad \text{on} \quad &\p\Omega. 
\end{align}

However, there are two difficult problems related to the equation posed on a bounded domain.
\begin{itemize}
    \item \textbf{Lack of the entropy estimate.} In the case of bounded domain, we cannot use entropy estimate as in \eqref{eq:entropy1}. This is because the nonlocal operator is defined as \eqref{operatorB_2} rather than \eqref{operatorB}. As a consequence, we cannot symmetrize the expression with gradients and obtain the term $\frac{1}{2\eps^{2}}\int_{0}^{t}\int_{\Td} \int_{\Td} \omega_{\varepsilon}(y) \, |\nabla u(x) - \nabla u(x-y)|^2$ in the dissipation of the entropy.
    \item \textbf{Recovery of the Neumann boundary conditions.} The question is whether we can prove that in the limit $\f{\p u}{\p\vec{n}}=0$ on $\p\Omega$. This is possible for the equation with constant mobility. More precisely, in \cite{MR4248454}, Authors were discussing the problem of nonlocal to local convergence for the Cahn-Hilliard equation with constant mobility. The constant mobility allows to obtain uniform bound on $\|B_{\eps}(u_{\eps})\|_{2}$ which allows to conclude that $\f{\p u}{\p\vec{n}}=0$ on $\p\Omega$. This is an extremely interesting phenomenon as this new boundary condition appears only in the limit. In our case, the estimate $\|B_{\eps}(u_{\eps})\|_{2}$ seems unavailable.
\end{itemize}

\noindent A possible approach to overcome this problem is to apply Serfaty-Sandier approach on the convergence of gradient flows \cite{MR2082242, MR2836361}.

\section{Existence of weak solutions to the nonlocal problem}

The existence of weak solutions for the local Cahn-Hilliard equation with degenerate mobility usually follows the method from~\cite{MR1377481}. The idea is to apply a Galerkin scheme with a non-degenerate regularized mobility, \ie, calling $m(n)$ the mobility, then one considers an approximation $m_{\eps}(n)\ge\eps$. Finally, using standard compactness methods one can prove the existence of weak solutions for the initial system. The uniqueness of the weak solutions is still an open question.\\

In the case of the nonlocal Cahn-Hilliard equation, we have to rely on a fixed point method. We first consider a nondegenerate mobility, and the fixed point argument is put on the nonlocality. Then, we pass to the limit to obtain the nonlocal Cahn-Hilliard equation with degenerate mobility.
\subsection{Approximating solutions}
Following the scheme above, we first consider a nondegenerate mobility and prove the existence of the following system
 \begin{alignat}{2}
\partial_t u_\delta &= \DIV(T_{\delta}(u_\delta) \nabla \mu_\delta)\quad &&\quad (0,+\infty)\times \Td, \label{eq:approx_problem_1}\\
\mu_\delta &= B_{\varepsilon}[u_\delta] + F'(u_{\delta})\quad &&\quad (0,+\infty)\times \Td \label{eq:approx_problem_2},
 \end{alignat}
where $\delta>0$ is a small parameter such that $2\delta < \frac{1}{\delta} - 1$, $\delta < \frac{1}{4}$ and
\begin{equation}
\label{regmob}
T_{\delta}(u)=
\begin{cases}
\delta\quad&\text{for $u\le \delta$,}\\
\mbox{smooth monotone interpolation}\quad&\text{for $u \in [\delta, 2 \delta]$,}\\
u \quad&\text{for $u \in [2 \delta, \f{1}{\delta} - 1]$,}\\
\mbox{smooth monotone interpolation}\quad&\text{for $u \in [\frac{1}{\delta} - 1, \f{1}{\delta}]$,}\\
\f 1 \delta\quad&\text{for $u\ge\f 1 \delta$.}
\end{cases}
\end{equation}
The estimates for the sequence $\{u_{\delta}\}$ will be obtained from the dissipation of energy and entropy. The definition of energy $E_{\varepsilon}$ remains the same as in \eqref{eq:energy1}. However, the definition of entropy has to be adapted to take into account the fact that we don't know if the solution remains nonnegative. To this end, we define a function $\phi_\delta$ by an explicit formula
\begin{equation}
\label{eq:entropycases}
\phi_{\delta}(x) = \int_1^x \int_1^y \frac{1}{T_{\delta}(z)} \diff z \diff y.
\end{equation}
\begin{lem}\label{lem:properties_phi_delta}
Let $\phi_{\delta}$ be defined with \eqref{eq:entropycases} and $\phi(x) = x(\log(x) -1) + 1$. Then, 
\begin{enumerate}[label=(P\arabic*)]
    \item $\phi_{\delta}''(x)=\f{1}{T_{\delta}(x)}$ and $\phi_{\delta}(1) = \phi_{\delta}'(1) = 0$,
    \item $\phi_\delta(x) \to \phi(x)$ for $x \geq 0$ as $\delta\to 0$,
    \item $\phi_\delta(x) \geq 0$ for all $x \in \R$,
    \item\label{item_bound_on_phi_delta_interms_of_phi} $\phi_\delta(x) \leq \phi(x) + \frac{\delta}{2(\delta - 1)} x^2 + 3$ for $x \geq 0$,
    \item\label{item_blow_up_phi_delta} $\phi_{\delta}(x) \to \infty$ when $\delta \to 0$ for all $x < 0$.
\end{enumerate} 
\end{lem}
The proof is presented in Appendix \ref{subsect:proof_phi_delta}.

\begin{thm}
\label{thm:existence_approx_solutions}
Let $\delta > 0$, $\varepsilon_0$ be as in \eqref{eq:eps_0_restriction} and $F \in C^4$. For $\varepsilon<\varepsilon_0$ there exists classical solution \eqref{eq:approx_problem_1}--\eqref{eq:approx_problem_2}. Moreover, they satisfy the mass, energy, and entropy conservation: for all $t>0$
\begin{align}
&\int_{\Td}u_{\delta}(t,\cdot)\diff x=\int_{\Td}u^{0}\diff x,\label{PWreg}\\
&E_{\varepsilon}[u_\delta](t) + \int_{0}^{t}\int_{\Td} T_{\delta}(u_{\delta}) \, |\nabla \mu_\delta|^2 = E_{\varepsilon}[u^{0}],\label{eq:energyreg}\\
&\Phi_{\delta}[u_\delta](t) + \frac{1}{2\eps^{2}}\int_{0}^{t}\int_{\Td} \int_{\Td} \omega_{\varepsilon}(y) \, |\nabla u_{\delta}(x) - \nabla u_{\delta}(x-y)|^2 + \int_0^t \int_{\Td} F''(u_{\delta}) |\nabla u_{\delta}|^{2} = \Phi_{\delta}[u^{0}].\label{eq:entropyreg}
    \end{align}
\end{thm}

\begin{thm}\label{thm:estim_reg}
Let $\varepsilon_0$ be as in \eqref{eq:eps_0_restriction}. Let $F$ satisfy Assumption \ref{ass:potentialF} with an additional constraint $2\,C_{10} < C_p$. Then, the following sequences are bounded uniformly in $\delta \in (0,1)$ and $\varepsilon \in (0,\varepsilon_0)$
\begin{enumerate}[label=(A\arabic*)]
    \item\label{item_estreg1} $\{u_\delta\}_{\delta}$ in $L^{\infty}(0,T; L^{k}( \Td))$,
    \item\label{item_estreg_add} $\{u_\delta\}_{\delta}$ in $L^{k}(0,T; L^{k \frac{d}{d-2}}( \Td))$,
    \item\label{item_estreg3} $\{\sqrt{T_\delta(u)}\, \nabla \mu_\delta\}_{\delta}$ in $L^2((0,T)\times \Td)$,
    \item\label{item_estreg2} $\{\nabla u_\delta\}_{\delta}$ in $L^2((0,T)\times \Td)$,
    \item\label{item_estreg4} $\{\partial_t u_\delta\}_{\delta}$ in $L^2(0,T; W^{-1,s'}(\Td))$,
    \item\label{item_estreg5} $\{\partial_t \nabla u_\delta\}_{\delta}$ in $L^2(0,T; W^{-2,s'}(\Td))$,
    \item\label{item_estreg6} $\{\sqrt{F_{1}''(u_\delta)}\nabla u_\delta\}_{\delta}$ in $L^2((0,T)\times \Td)$,
    \item\label{item_estreg7} $\Phi_{\delta}[u_\delta]$ in $L^{\infty}(0,T)$,
\end{enumerate}
where $k$ and $s$ have been defined in Notation \ref{not:s_and_k}. 
\end{thm}

To prove Theorem \ref{thm:existence_approx_solutions}, we need to assume that $F \in C^4$ which allows us to use known results about classical solutions to uniformly parabolic equations.

\begin{proof}[Proof of Theorem \ref{thm:existence_approx_solutions}]
As $\delta>0$ is fixed in this result, we write $u$ instead of $u_{\delta}$. Given $w$ we consider an auxiliary equation
\begin{equation}\label{eq:eq_for_Schauder_FP}
\partial_t u = \DIV\left(T_{\delta}(u) \nabla u \left(\frac{1}{\varepsilon^2} + F''(u) \right)\right) -  \DIV\left(T_{\delta}(u) \frac{w\ast \nabla \omega_{\varepsilon}}{\varepsilon^2} \right). 
\end{equation}

Let $\alpha,\sigma, M,\kappa$ be parameters to be specified later. We want to apply Schauder fixed point theorem to the map
\begin{equation*}
\begin{array}{cccccc}
P & : & X & \to & X & \\
P & : & w & \mapsto & u & \text{solution of~\eqref{eq:eq_for_Schauder_FP}}, \\
\end{array}
\end{equation*}
where $X$ is defined as the set 
$$
X=\{w\in C^{\alpha,\alpha/2}([0,T] \times \Td),\norm{w}_{\infty,\sigma}\le M\}
$$
with the norm
$$
\|w\|_{X} :=\norm{w}_{\infty,\sigma}+\kappa\norm{w}_{\alpha,\alpha/2}
$$
and the norm $\norm{\cdot}_{\alpha,\alpha/2}$ is the usual Hölder seminorm in space-time. We also define
\begin{equation}\label{eq:norm_sigma}
\norm{w}_{\infty,\sigma}:=\sup_{[0,T]\times \Td} |u(t,x)| e^{-\sigma t}.
\end{equation}
Note that the new norm is equivalent to the usual supremum norm so all topological properties do not change. We need to prove that $P$ is continuous, $P$ maps in fact $X$ to $X$, and that $P(X)$ is relatively compact in $X$. First, we prove that $P(w)=u$ is the unique classical solution of equation~\eqref{eq:eq_for_Schauder_FP} so that $P$ is well defined and find Hölder estimates which will be useful to prove the continuity of the operator as well as its relative compactness.

\underline{\textit{Step 1: $P$ is well defined and Hölder estimates}}. Equation~\eqref{eq:eq_for_Schauder_FP} is equivalent to saying that $u$ solves parabolic equation
$$
\partial_t u = \DIV A(t,x,u,\nabla u) + B(t,x,u,\nabla u), \qquad \qquad (u)_{\Td} = (u_0)_{\Td}
$$
with
$$
A(t,x,z,p) = T_{\delta}(z)\,p\,\left(\frac{1}{\varepsilon^2} + F''(z) \right), \qquad  B(t,x,z,p)=-T_{\delta}'(z) \, p \cdot \frac{w \ast \nabla \omega_{\varepsilon}}{\varepsilon^2} - T_{\delta}(z) \, \frac{w\ast \Delta \omega_{\varepsilon}}{\varepsilon^2},
$$
and we recall that $w\in X$ is Hölder continuous. The function $A$ satisfies the strong parabolicity condition for sufficiently small $\varepsilon>0$, i.e.
$$
A(t,x,z,p) \cdot p \geq \delta \, p^2 \, \frac{1}{2\,\varepsilon_0^2}
$$
for all $\varepsilon < \varepsilon_0$ (this uses Assumptions \ref{assumption_pot_3}, \ref{assumption_pot_4} and \eqref{eq:eps_0_restriction}). Since the derivatives $A_p$, $A_z$, $A_t$, $A_x$ and function $B$ are Hölder continuous as functions of $(t,x,z,p)$, \cite[Theorems 12.10, 12.14]{MR1465184} asserts that there exists a unique classical solution to \eqref{eq:eq_for_Schauder_FP} such that
$$
\|u\|_{C^{1 + \alpha, 1+\alpha/2}} \leq C(\delta, \varepsilon_0, \|w\|_{C^{ \alpha, \alpha/2}}).
$$
With this estimate, \eqref{eq:eq_for_Schauder_FP} can be considered as a linear equation so that the linear theory for parabolic equations \cite[Theorem 5.14]{MR1465184} implies
\begin{equation}\label{estimate:lieberman}
\|u\|_{C^{2+\alpha, 1+\alpha/2}} \leq C(\delta, \varepsilon_0, \|w\|_{C^{\alpha, \alpha/2}}).
\end{equation}

Therefore $u$ is a classical solution of~\eqref{eq:eq_for_Schauder_FP} and it admits the Hölder bound~\eqref{estimate:lieberman}.

\underline{\textit{Step 2: The operator $P$ is continuous.}}. We consider a sequence $\{w_{n}\}_{n}$ in $X$ such that $\|w_{n} - w\|_X \to 0$. Then $u_{n} = P(w_{n})$ is compact in $C^{2,1}$ from estimate~\eqref{estimate:lieberman} and Arzela-Ascoli. We choose subsequence such that $u_{n_{k}} \to u$ in $C^{2,1}$. These functions satisfy
\begin{equation}\label{eq:PDE_sat_by_unk}
\partial_t u_{n_{k}} = \DIV\left(T_{\delta}(u_{n_{k}}) \nabla u_{n_{k}} \left(\frac{1}{\varepsilon^2} + F''(u_{n_{k}}) \right)\right) -  \DIV\left(T_{\delta}(u_{n_{k}}) \frac{w_{n_{k}}\ast \nabla \omega_{\varepsilon}}{\varepsilon^2} \right). 
\end{equation}

Passing to the limit in \eqref{eq:PDE_sat_by_unk} and using uniqueness of solutions to \eqref{eq:eq_for_Schauder_FP} from~\cite{MR1465184}, we obtain that for every subsequence of $\{u_{n}\}_{n}$ we can extract a subsequence which converges to a unique limit $u = P(w)$. By a standard subsequence argument, this means that the whole sequence $\{u_{n}\}_{n}$ converges to $u=P(w)$. Therefore $P$ is continuous.

\underline{\textit{Step 3: $P$ maps $X$ to $X$}}. We write the equation~\eqref{eq:eq_for_Schauder_FP} in the form
\begin{align*}
\partial_t u &=T_{\delta}'(u) |\nabla u|^2 \left(\frac{1}{\varepsilon^2} + F''(u) \right) + T_{\delta}(u) \Delta u \left(\frac{1}{\varepsilon^2} + F''(u) \right) + T_{\delta}(u) |\nabla u|^2 F^{(3)}(u)
\\
& \quad -  T_{\delta}'(u)\, \nabla u \cdot \frac{w\ast \nabla \omega_{\varepsilon}}{\varepsilon^2}  - T_{\delta}(u) \frac{w\ast \Delta \omega_{\varepsilon}}{\varepsilon^2}. 
\end{align*}
We substitute $u = v\, e^{\sigma t}$ and we compute PDE satisfied by $v$: 
\begin{align*}
\partial_t v \, e^{\sigma t} + v \, \sigma \, e^{\sigma t} &=T_{\delta}'(u) |\nabla v|^2 \left(\frac{1}{\varepsilon^2} + F''(u) \right) \, e^{2\sigma t} + T_{\delta}(u) \Delta v \left(\frac{1}{\varepsilon^2} + F''(u) \right) e^{\sigma t}
\\
& \quad  + T_{\delta}(u) |\nabla v|^2 F^{(3)}(u) e^{2\sigma t} \, -  T_{\delta}'(u)\, \nabla v \cdot \frac{w\ast \nabla \omega_{\varepsilon}}{\varepsilon^2}e^{\sigma t}  - T_{\delta}(u) \frac{w\ast \Delta \omega_{\varepsilon}}{\varepsilon^2}. 
\end{align*}
Now, we multiply by $v$ and evaluate the equation at the point $(t_{*},x_{*})$ where $v^2$ attains its maximum. Therefore, all the terms with $\nabla v$ and $|\nabla v|^2$ vanish (as $|\nabla v|^2 v = \nabla v \cdot \nabla v^2/2$).
\begin{align*}
\frac{1}{2}\partial_t v^2 e^{\sigma t_{*}} + v^2\sigma e^{\sigma t_{*}} = T_{\delta}(u)\, v\, \Delta v \left(\frac{1}{\varepsilon^2} + F''(u) \right)e^{\sigma t_{*}}  - v\, T_{\delta}(u) \frac{w\ast \Delta \omega_{\varepsilon}}{\varepsilon^2}. 
\end{align*}
Using $v\, \Delta v = -|\nabla v|^2 + \Delta v^2 \leq 0$ and $\partial_t v^2 \geq 0$ we obtain
$$
v^2 \, \sigma \, e^{\sigma t_{*}} \leq -v\, T_{\delta}(u) \frac{w\ast \Delta \omega_{\varepsilon}}{\varepsilon^2},
$$
so that
$$
v^2(t_*,x_*) \, \sigma \, e^{\sigma t_{*}} \leq |v(t_*,x_*)| \frac{\|\Delta \omega_{\varepsilon}\|_{1} \|w(t_*,\cdot)\|_{\infty}}{\delta \varepsilon^2}. 
$$
where we used the definition of $T_{\delta}$. As $v^2$ attains maximum at $(t^*,x^*)$, $|v(t_*,x_*)|$ also attains maximum at $(t^*,x^*)$. Therefore, taking into account the initial condition
$$
\|v\|_{\infty} \,  \leq \max \left(\frac{\|\Delta \omega_{\varepsilon}\|_{1} \|w(t_*,\cdot)\|_{\infty}}{\delta \varepsilon^2 \sigma} e^{-\sigma t_{*}}, \|u_0\|_{\infty}\right) \leq
\max\left(\frac{\|\Delta \omega_{\varepsilon}\|_{1} \|w e^{-\sigma t}\|_{\infty}}{\delta \varepsilon^2 \sigma}, \|u_0\|_{\infty}\right).
$$
Choosing $\sigma = 2 \|\Delta \omega_{\varepsilon}\|_{1}/(\delta \varepsilon^2)$, we obtain estimate
$$
\|v\|_{\infty} \,  \leq \max\left(\frac{1}{2}\, \|w e^{-\sigma t}\|_{\infty}, \|u_0\|_{\infty}\right).
$$
By definition of the norm 
\begin{equation}\label{eq:estimate_operator_P}
\|Pw\|_{\infty,\sigma} \leq \max\left(\frac{1}{2}\, \|w\|_{\infty,\sigma}, \|u_0\|_{\infty}\right).
\end{equation}

Moreover, the parabolic version of de Giorgi-Nash-Moser theory, see~\cite[Chap. V, Theorem 1.1]{ladyzhenskaya1968linear}, implies that there exists $\alpha= \alpha(\|w\|_{\infty,\sigma})$ such that the solution of~\eqref{eq:eq_for_Schauder_FP} satisfy 
$$
\|u\|_{C^{\alpha,\alpha/2}} \leq f(\|w\|_{\infty,\sigma}).
$$
Without loss of generality we may assume that $f(\|w\|_{\infty,\sigma})$ does not decrease and $\alpha(\|w\|_{\infty,\sigma})$ does not increase when $\|w\|_{\infty,\sigma}$ increases.\\

\noindent We proceed to choosing values of parameters $M$, $\alpha$, $\kappa$ and concluding the proof. We choose
$$
M=3\norm{u_{0}}_{L^{\infty}},
\qquad
\alpha = \alpha(M), 
\qquad 
\kappa=\f{\norm{u_{0}}_{L^{\infty}}}{2f(M)}.
$$
Since $w$ is in $X$ and $f$ is nondecreasing we obtain
\begin{equation*}
\norm{Pw}_{X}\le \f{1}{2}\|w\|_{\infty,\sigma}+ \|u_0\|_{\infty}+\kappa f(\|w\|_{\infty,\sigma}) \le \f{M}{2}+ \|u_0\|_{\infty}+\kappa f(M) \le 3\norm{u_{0}}_{L^{\infty}}=M.  
\end{equation*}
This means that $P$ maps $X$ to $X$.

\underline{\textit{Step 4: $P(X)$ is relatively compact in $X$}}. The relative compactness of $P(X)$ follows from \eqref{estimate:lieberman}. 

The proof is concluded. 
\end{proof}

\begin{proof}[Proof of Theorem \ref{thm:estim_reg}]
To prove \ref{item_estreg1} and \ref{item_estreg3} we want to apply~\eqref{eq:energyreg} and Assumption \ref{ass:potentialF} on the potential. The energy identity reads:
$$
\int_{\Td}F(u_{\delta})\diff x+\f{1}{4\eps^{2}}\int_{\Td}\int_{\Td}\omega_{\eps}(y)|u_{\delta}(x)-u_{\delta}(x-y)|^{2}\diff x\diff y
+ \int_{0}^{t}\int_{\Td} T_{\delta}(u_{\delta}) \, |\nabla \mu_\delta|^2 = E_{\varepsilon}[u^{0}],
$$
Applying Lemma \ref{lem:Poincare_with_average}, we deduce
$$
\int_{\Td}F(u_{\delta})\diff x+C_p \int_{\Td} |u - (u)_{\Td}|^2
+ \int_{0}^{t}\int_{\Td} T_{\delta}(u_{\delta}) \, |\nabla \mu_\delta|^2 \le E_{\varepsilon}[u^{0}]
$$
Splitting $F = F_1 +F_2$ and applying Assumption \ref{ass:potentialF} we obtain 
$$
\int_{\Td}F_1(u_{\delta})\diff x+C_p \int_{\Td} |u_{\delta} - (u_{\delta})_{\Td}|^2
+ \int_{0}^{t}\int_{\Td} T_{\delta}(u_{\delta}) \, |\nabla \mu_\delta|^2 \le E_{\varepsilon}[u^{0}] + C_{9} + C_{10} \int_{\Td} |u_{\delta}|^2 
$$
Note that by conservation of mass, $(u_{\delta})_{\Td} = (u^0)_{\Td}$. Therefore, applying the simple inequality $|a+b|^2 \leq 2|a|^2 + 2|b|^2$ and $C_p > 2\,C_{10}$, we obtain an $L^{\infty}(0,T; L^2(\Td))$ estimate on $\{u_{\delta}\}$ which can be improved to $L^{\infty}(0,T; L^k(\Td))$ if $F_1 \neq 0$ cf. \ref{assumption_pot_3} in Assumption \ref{ass:potentialF}. Then, \ref{item_estreg1} and so, \ref{item_estreg3} is easily implied by the energy as all possibly negative terms are bounded. \\

\noindent Now, to prove \ref{item_estreg2} we want to use the entropy equality \eqref{eq:entropyreg}:
$$
\Phi_{\delta}[u_\delta](t) + \frac{1}{2\eps^{2}} \int_0^t \int_{\Td} \int_{\Td} \omega_{\varepsilon}(y) \, |\nabla u_{\delta}(x) - \nabla u_{\delta}(x-y)|^2 + \int_0^t \int_{\Td} F''(u_{\delta}) |\nabla u_{\delta}|^{2} = \Phi_{\delta}[u^{0}].
$$
To exploit it, for $\gamma$ to be chosen later, $\varepsilon \in (0, \widetilde{\varepsilon}_0(\gamma))$ we have by Lemma~\ref{lem:poincare_nonlocal_H1_L2}
\begin{multline*}
\Phi_{\delta}[u_\delta](t) + \frac{1}{\gamma} \int_0^t \int_{\Td} |\nabla u_\delta |^2 + \int_0^t \int_{\Td} F_1''(u_{\delta}) |\nabla u_{\delta}|^{2} \le \\ \le \Phi_{\delta}[u^{0}] +  C(\gamma) \int_0^t \int_{\Td} \|u_{\delta} \|^2_{L^2(\Td)} + \|F_2''\|_{\infty} \int_0^t \int_{\Td} |\nabla u_{\delta}|^{2}.
\end{multline*}
We choose $\gamma = \frac{1}{1+\|F_2''\|_{\infty}}$ which yields estimates \ref{item_estreg2}, \ref{item_estreg6} and \ref{item_estreg7} (here, we also exploit \ref{item_bound_on_phi_delta_interms_of_phi} in Lemma \ref{lem:properties_phi_delta} to control $\Phi_{\delta}[u^{0}]$). Now, to see~\ref{item_estreg4} we take a smooth test function $\varphi$ and write thanks to the Hölder inequality
\begin{align*}
\left|\int_0^T \int_{\Td} \partial_t u_{\delta} \, \varphi \diff x \diff t \right| &= \left|\int_0^T \int_{\Td} T_{\delta}(u_{\delta})^{1/2}T_{\delta}(u_{\delta})^{1/2}\nabla\mu_{\delta}\cdot\nabla\varphi \diff x \diff t \right|\\
&\le \|T_{\delta}(u_{\delta})^{1/2}\|_{L^{\infty}(0,T;L^{2k}(\Td))}\|T_{\delta}(u_{\delta})^{1/2}\nabla\mu_{\delta}\|_{L^{2}((0,T)\times \Td)}\|\nabla\varphi\|_{L^{2}(0,T;L^{s}(\Td))}\\
&\le C\|\nabla\varphi\|_{L^{2}(0,T;L^{s}(\Td))}. 
\end{align*}
In the last line we used estimates~\ref{item_estreg1}, \ref{item_estreg3} and the definition of $T_{\delta}$. This concludes the proof for estimates~\ref{item_estreg4} and then~\ref{item_estreg5} easily follows.\\

\noindent Finally, we prove \ref{item_estreg_add}. We note from \ref{item_estreg6} that $\{\nabla u_{\delta}^{k/2}\}$ is bounded in $L^{2}(0,T; L^2(\Td))$ and from \ref{item_estreg1} that $\{u_{\delta}^{k/2}\}$ is bounded in $L^{\infty}(0,T; L^2(\Td))$. Therefore, by Sobolev embedding, we obtain that $\{u_{\delta}^{k/2}\}$ is bounded in $L^2(0,T; L^{\frac{2d}{d-2}}(\Td))$ so that $\{u_{\delta}\}$ is bounded in $L^k(0,T; L^{k \frac{d}{d-2}}(\Td))$.
\end{proof}

\subsection{Proof of Theorem~\ref{thm:weaksoldelta}}

\begin{proof}[Proof of Theorem \ref{thm:weaksoldelta}]
\underline{\textit{Step 1: Approximation of the potential.}}

For $F$ as in Assumption \ref{ass:potentialF}, we consider its mollification $F_{\delta} = F \ast \eta_{\delta}$ where $\{\eta_{\delta}\}$ is the usual mollifier. We note that $F_{\delta}$ is $C^4$ and that $F$, $F_{\delta}$ satisfy Assumption \ref{ass:potentialF} with comparable constants $C_{1}, ..., C_{10}$, see Lemma \ref{lem:F_mollification_estimates}. The most important is constant $C_{10}$ because there is a constraint on it in terms of $C_p$. More precisely, $F$ satisfies Assumption \ref{ass:potentialF} with $C_{10} < C_p/4$ so that from Lemma \ref{lem:F_mollification_estimates} we have that $F_{\delta}$ satisfies it with $2C_{10} < C_p/2$. This allows to apply Theorem \ref{thm:estim_reg} to otain uniform estimates. Moreover $F_\delta=F_{\delta,1}+F_{\delta,2}$ with $F^{(p)}_{\delta,(1,2)}\xrightarrow[\delta\to 0]{pointwise}F^{(p)}_{(1,2)}$ where $p=0,1,2$ is the order of derivative. 

\noindent \underline{\textit{Step 2: Compactness.}} Using Theorem \ref{thm:existence_approx_solutions}, we can obtain $u_{\delta}$ such that for all $\varphi \in L^{2}(0,T;W^{1,\infty}(\Td))$
\begin{equation}\label{eq:weaksolregmob}
\begin{split}
    \int_{0}^{T}\langle\p_{t}u_{\delta},\varphi\rangle_{(W^{1,s}(\Td))',W^{1,s}(\Td)}
&+ \\ +
\int_{0}^{T}\int_{\Td}T_{\delta}&(u_{\delta})\nabla B_{\varepsilon}[u_{\delta}]\cdot\nabla\varphi
+
\int_{0}^{T}\int_{\Td}u_{\delta}F_{\delta}''(u_{\delta})\nabla u_{\delta}\cdot\nabla\varphi=0.
\end{split}
\end{equation}
The plan is to send $\delta \to 0$ in \eqref{eq:weaksolregmob}. By Theorem \ref{thm:estim_reg} and standard compactness results we can extract a subsequence (not relabelled) such that 
\begin{enumerate}[label=(B\arabic*)]
\item\label{item_comp_1} $u_{\delta}\to u$ a.e. and in $L^2((0,T)\times\Td) \cap L^k((0,T)\times\Td)$,
\item\label{item_comp_2} $\nabla u_\delta \rightharpoonup\nabla u$ in $L^2((0,T)\times\Td)$,
\item\label{item_comp_3} $\partial_t u_{\delta} \weak \partial_t u$ in $L^2(0,T; W^{-1,s'}(\Td))$,
\item\label{item_comp_4} $\sqrt{F_{1,\delta}''(u_{\delta})}\nabla u_{\delta}\rightharpoonup\xi$ in $L^2((0,T)\times\Td)$ for some $\xi\in L^2((0,T)\times\Td)$.
\end{enumerate}
Only \ref{item_comp_1} needs some justification. From~\ref{item_estreg1}, \ref{item_estreg2}, \ref{item_estreg4} and Aubin-Lions lemma, we obtain the strong convergence  $u_{\delta}\to u$ a.e. and in $L^2((0,T)\times\Td)$. To see the second strong convergence, we interpolate between $L^{\infty}(0,T; L^k(\Td))$ and $L^k(0,T; L^{k \frac{d}{d-2}}(\Td))$ to prove that $\{u_{\delta}\}$ is bounded in $L^{k+\kappa}(0,T; L^{k+\kappa}(\Td))$ for some $\kappa>0$ because $k\frac{d}{d-2}>k$. Now, interpolating between $L^{k+\kappa}(0,T; L^{k+\kappa}(\Td))$ and $L^2((0,T)\times\Td)$ we obtain strong convergence in $L^{k}((0,T)\times \Td)$.

\underline{\textit{Step 3: Nonnegativity of $u$.}} 
The plan is to obtain a contradiction with the uniform estimate of the entropy. For $\alpha >0$, we define the sets
\begin{equation*}
  V_{\alpha,\delta}=\{(t,x)\in (0,T)\times \Td: \,  u_{\delta}(t,x)\le-\alpha \},
\end{equation*}
\begin{equation*}
  V_{\alpha,0}=\{(t,x)\in (0,T)\times \Td: \,  u(t,x)\le-\alpha \}.
\end{equation*}
By nonnegativity of $\phi_\delta$ (see \eqref{eq:entropycases} as well as the properties below) and \ref{item_estreg7} in Theorem~\ref{thm:estim_reg}, there is a constant $C(T)$ such that
\begin{equation*}
 \int_{V_{\alpha,\delta}} \phi_\delta(u_\delta) \diff x \diff t \leq   \int_{(0,T)\times \Td} \phi_\delta(u_\delta) \diff x \diff t \le C(T).
\end{equation*}
For $u_\delta \leq -\alpha$, we have $0 \leq \phi_\delta(-\alpha) \leq \phi_\delta(u_\delta)$ because $\phi_{\delta}'(x) \leq 0$ for $x \leq 0$, see \eqref{eq:entropycases}. Therefore,
$$
0 \leq \phi_{\delta}(-\alpha) \, \int_{V_{\alpha,\delta}} 1 \diff x \diff t = \int_{V_{\alpha,\delta}} \phi_\delta(x) \diff x \diff t   \leq C(T).
$$
\noindent Sending $\delta \to 0$, exploiting \ref{item_blow_up_phi_delta} in Lemma \ref{lem:properties_phi_delta} and using the strong convergence of $u_{\delta} \to u$ we discover
$$
\int_{V_{\alpha, 0}} 1 \diff x \diff t = \lim_{\delta \to 0} \int_{V_{\alpha, \delta}} 1\diff x \diff t = 0
$$
(we use here the fact from measure theory asserting that on the measure space $(X,\mu)$ if $f_n, f: X \to \R$ and $f_n \to f$ in $L^1(X,\mu)$ then for $\alpha \in \R$ we have $\int_{f_n < \alpha} \diff \mu \to \int_{f < \alpha} \diff \mu$ as $n \to \infty$). This means that $V_{\alpha,0}$ is a null set for each $\alpha > 0$, concluding the proof of the nonnegativity.

\underline{\textit{Step 4: Identification $\xi = \sqrt{F_1''(u)} \nabla u$.}} We want to use~\ref{item_comp_4} so we have to identify $\xi$. For that purpose, we use the convergence a.e. of $u_\delta$ in~\ref{item_comp_1} and the pointwise convergence $F_{\delta,1}'' \to F_{1}''$ to deduce that $F_{\delta,1}(u_{\delta}) \to F_1(u)$ a.e. Next, using Assumption~\ref{assumption_pot_3} for $F_{\delta,1}$ and estimate~\ref{item_estreg1}
$$
\left|\sqrt{F_{\delta,1}''(u_\delta)}\right|^{2} \leq C_{3} |u_{\delta}|^{k-2} + C_{4}.
$$
As (RHS) is uniformly integrable by strong convergence \ref{item_comp_1}, we deduce that $\left|\sqrt{F_{\delta,1}''(u_\delta)}\right|^{2}$ is uniformly integrable so that the Vitali convergence theorem implies
\begin{equation*}
\sqrt{F_{\delta,1}''(u_\delta)}\to \sqrt{F_{1}''(u)}\quad \text{in } L^2((0,T)\times \Td)   
\end{equation*}
Using weak convergence of gradient \ref{item_comp_2}, we finally obtain $\xi= \sqrt{F_{1}''(u)}\nabla u$.

\underline{\textit{Step 5: Passing to the limit in the first two terms of~\eqref{eq:weaksolregmob}}}. Using~\ref{item_comp_3} it is easy to pass to the limit in the first term of~\eqref{eq:weaksolregmob}. 
Now we focus on the second term. Note that
$$
\nabla B_{\varepsilon}[u_{\delta}](x) =\frac{1}{\varepsilon^2}( \nabla u_{\delta} -  \omega_{\varepsilon} \ast \nabla u_{\delta}).
$$

The two terms of $\nabla B_{\eps}$ are treated in the same way. We focus only on the harder term $\nabla u_{\delta}$ which does not have regularizing properties of the convolution. For this term it is sufficient to prove that $T_{\delta}(u_\delta)\nabla u_\delta\rightharpoonup u\nabla u$ weakly in $L^{2}(0,T;L^{1}(\Td))$. We first note that by definition of $T_\delta$, the strong convergence~\ref{item_comp_1} and the nonnegativity of $u$, we obtain $T_{\delta}(u_\delta) \to u$ strongly in $L^2((0,T)\times\Td)$. Hence, the result follows from weak convergence of the gradient~\ref{item_comp_2}.

\underline{\textit{Step 6: Passing to the limit in the third term of~\eqref{eq:weaksolregmob}}}. For the third term we write $F_\delta''=F''_{\delta,1}+F''_{\delta,2}$ as discussed in Step 1. Then we decompose 
\begin{align*}
\int_{0}^{T}\int_{\Td}T_{\delta}(u_{\delta})F_{\delta}''(u_{\delta})\nabla u_{\delta}\cdot\nabla\varphi&=\int_{0}^{T}\int_{\Td}T_{\delta}(u_{\delta})F_{\delta,1}''(u_{\delta})\nabla u_{\delta}\cdot\nabla\varphi+\int_{0}^{T}\int_{\Td}T_{\delta}(u_{\delta})F_{\delta,2}''(u_{\delta})\nabla u_{\delta}\cdot\nabla\varphi\\
&=I_{1}+I_{2}. 
\end{align*}
For $I_{1}$ we write
\begin{equation*}
I_{1}=    \int_{0}^{T}\int_{\Td}T_{\delta}(u_{\delta})\sqrt{F_{\delta,1}''(u_{\delta})}\sqrt{F_{\delta,1}''(u_{\delta})}\nabla u_{\delta}\cdot\nabla\varphi. 
\end{equation*}
It remains to prove that $T_{\delta}(u_{\delta})\sqrt{F_{\delta,1}''(u_\delta)}$ converges strongly in $L^2((0,T)\times\Td)$. Note that since $u_\delta\to u\ge 0$ we have $T_{\delta}(u_{\delta})\sqrt{F_{\delta,1}''(u_\delta)}\to u\sqrt{F_{1}''(u)}$ a.e. Moreover, 
$$
\left(T_{\delta}(u_{\delta})\sqrt{F_{\delta,1}''(u_\delta)}\,\right)^2 \leq C_{3}\,|u_{\delta}|^{k} + C_{4}
$$
As the (RHS) is uniformly integrable by strong convergence, we deduce that (LHS) is uniformly integrable. Hence, the Vitali convergence theorem implies  Assumption~\ref{assumption_pot_3} and Estimate~\ref{item_estreg1} show that 
\begin{equation*}
T_{\delta}(u_{\delta})\sqrt{F_{\delta,1}''(u_\delta)} \to u \sqrt{F_{1}''(u)} \mbox{ in } L^2((0,T)\times \Td)
\end{equation*}
so that $
 I_{1}\to \int_{0}^{T}\int_{\Td}uF_{1}''(u)\nabla u\cdot\nabla\varphi.$
For $I_{2}$, as $\nabla u_{\delta} \rightharpoonup \nabla u$ weakly in $L^2((0,T)\times\Td)$, it is sufficient to prove the strong convergence of $T_{\delta}(u_\delta)F_{\delta,2}''(u_\delta)$ in $L^2((0,T)\times\Td)$. Thanks to Assumption~\ref{assumption_pot_4} on $F_{\delta,2}''$, this term is uniformly bounded so that trivially $|T_{\delta}(u_\delta)F_{\delta,2}''(u_\delta)| \leq \|F_{2}'' \|_{\infty} |T_{\delta}(u_{\delta})|$.
Therefore, Vitali convergence theorem implies $T_{\delta}(u_\delta)F_{\delta,2}''(u_\delta)$ in $L^2((0,T)\times\Td)$ and so
\begin{equation*}
 I_{2}\to \int_{0}^{T}\int_{\Td}uF_{2}''(u)\nabla u\cdot\nabla\varphi.   
\end{equation*}

\underline{\textit{Step 7: Energy and entropy estimates}}. We pass to the limit $\delta\to 0$ in~\eqref{eq:energyreg}-\eqref{eq:entropyreg}. With the above convergences and properties of the weak limit, we obtain the result. This ends the proof of Theorem~\ref{thm:weaksoldelta}. 
\end{proof}

Now that weak solutions of the nonlocal Cahn-Hilliard equation have been constructed for a given initial datum, it remains to prove the convergence of the nonlocal system to the local one. This is the purpose of the next section. 

\section{Limit $\eps\to 0$}\label{sec:conv_eps}

Weak solutions of the local Cahn-Hilliard equation are understood in the sense of Definition~\ref{def:weak_sol_limit}. In order to prove the convergence of the nonlocal system to these solutions, we first collect the necessary estimates uniform in $\eps$. Then we pass to the limit $\eps\to 0$ to conclude the proof of Theorem~\ref{thm:final}.

\subsection{Uniform estimates in $\eps$}

We recall that in the previous section we had obtained the energy and entropy inequalities as well as estimates uniform in $\eps$.
\begin{lem}[Mass, energy, entropy]
The following identities hold true
\begin{align}
&\int_{\Td}u_{\eps}(t,\cdot)\diff x=\int_{\Td}u^{0}\diff x,\label{PWreg2} \\
&\f{d}{dt} E[u_\eps] + \int_{\Td} u_{\eps} \, |\nabla \mu_\eps|^2 \le 0,\label{eq:energy}\\
&\f{d}{dt}\Phi[u_\eps] + \frac{1}{2\eps^{2}} \int_{\Td} \int_{\Td} \omega_{\varepsilon}(y) \, |\nabla u_{\varepsilon}(x) - \nabla u_{\varepsilon}(x-y)|^2 \diff x \diff y + \int_{\Td} F''(u_{\eps}) |\nabla u_{\eps}|^{2} \diff x\le 0.\label{eq:entropy}
\end{align}
\end{lem}

\begin{lem}[Uniform estimates]\label{lem:uniform_est_just_eps}
The following sequences are bounded:
\begin{enumerate}[label=(\Alph*)]
    \item\label{item_est1} $\{u_\eps\}_{\eps}$ in $L^{\infty}(0,T; L^k( \Td))$,
    \item\label{item_est1_add} $\{u_\eps\}_{\eps}$ in $L^{k}(0,T; L^{k\f{d}{d-2}}( \Td))$,
    \item\label{item_est2} $\{\nabla u_\eps\}_{\eps}$ in $L^2((0,T)\times \Td)$,
    \item\label{item_est3} $\{\sqrt{u_\eps}\, \nabla \mu_\eps\}_{\eps}$ in $L^2((0,T)\times \Td)$,
    \item\label{item_est4} $\{\partial_t u_\eps\}_{\eps}$ in $L^2(0,T; W^{-1,s'}(\Td))$,
    \item\label{item_est5} $\{\partial_t \nabla u_\eps\}_{\eps}$ in $L^2(0,T; W^{-2,s'}(\Td))$,
    \item\label{item_est6} $\{\sqrt{F_{1}''(u_\eps)}\nabla u_\eps\}_{\eps}$ in $L^2((0,T)\times \Td)$.
\end{enumerate}
\end{lem}

Our last ingredient for the proof of Theorem~\ref{thm:final} is about the compactness of $u_\eps$ and its gradient. 

\begin{lem}[Compactness]\label{lem:compactness-FK-tx}
Sequences $\{u_\eps\}_{\eps}$ and $\{\nabla u_\eps\}_{\eps}$ are strongly compact in $L^2((0,T)\times \Td)$.
\end{lem}
\begin{proof}
The compactness of $\{u_\eps\}$ follows from the Lions-Aubin lemma applied to estimates~\ref{item_est1},\ref{item_est2} and \ref{item_est4}. Then, for the compactness of $\{\nabla u_\eps\}$, the detailed proof is presented in Appendix~\ref{app:FK_tx}. Roughly speaking, from~\ref{item_est5} we know that the sequence is compact in time. Compactness in space follows from Theorem~\ref{thm:ponce_tx} together with the estimate provided by the entropy on the quantity:
$$
\frac{1}{4\eps^{2}} \int_0^T \int_{\Td} \int_{\Td} \omega_{\varepsilon}(y) \, |\nabla u_{\varepsilon}(x) - \nabla u_{\varepsilon}(x-y)|^2 \diff x \diff y \diff t \leq C.
$$
An application of the Fréchet-Kolmogorov theorem yields the result.  
\end{proof}

Now we are ready to prove our main theorem. 

\subsection{Proof of Theorem~\ref{thm:final}: convergence $\eps\to 0$}\label{subsec:conv_eps}

We want to pass to the limit $\eps\to 0$ in Equations~\eqref{eq:CH1}-\eqref{eq:CH2} and obtain weak solutions of the local Cahn-Hilliard equation. We have at most bounds on the gradient of $u_{\eps}$ and the limit equation has four derivatives. That means we need to mimic at the epsilon level integration by parts for nonlocal operators. For that purpose, we define the operator 
\begin{equation}\label{eq:nonlocal_gradient}
S_{\eps}[\varphi](x,y):=\f{\sqrt{\omega_{\eps}(y)}}{\sqrt{2}\eps}(\varphi(x-y)-\varphi(x)) 
\end{equation}
which has the following properties:
\begin{lem}\label{lem:S_properties} The operator $S_{\eps}$ satisfies:
\begin{enumerate}[label=(S\arabic*)]
    \item $S_{\eps}$ is a linear operator that commutes with derivatives with respect to $x$,
    \item\label{propS_product_rule} for all functions $f,g: \Td \to \R$ we have
    \begin{align*}S_{\eps}[fg](x,y)-S_{\eps}[f](x,y)g(x)&-S_{\eps}[g](x,y)f(x)= \\ &= \f{\sqrt{\omega_{\eps}(y)}}{\sqrt{2}\eps}[(f(x-y)-f(x))(g(x-y)-g(x))].
    \end{align*}
    \item\label{propS_nonneg} for all $u, \varphi \in L^2(\Td)$
    $$
    \langle B_{\eps}[u](\cdot),\varphi(\cdot)\rangle_{L^{2}(\Td)}=\langle S_{\eps}[u](\cdot,\cdot),S_{\eps}[\varphi](\cdot,\cdot)\rangle_{L^{2}(\Td \times \Td)}.
    $$
    \item\label{propS:conv_to_energy}
    if $\{u_{\varepsilon}\}$ is strongly compact in $L^2(0,T; H^1(\Td))$ and $\varphi \in L^{\infty}((0,T)\times \Td)$ we have
    $$
    \int_0^T \int_{\Td} \int_{\Td} (S_{\varepsilon}[u_{\varepsilon}])^2 \, \varphi(t,x) \to D\, \int_0^T \int_{\Td} |\nabla u(t,x)|^2 \, \varphi(t,x)
    $$
    where $D = \frac{1}{2}\, \int_{B_1} \omega(y) |y|^2 \diff y$.
\end{enumerate}
\end{lem}
\begin{proof}
The first one is trivial. For the second one, we just observe
\begin{align*}
&(f(x-y)-f(x))(g(x-y)-g(x)) =\\ & \qquad \qquad = -(f(x-y) - f(x))\,g(x) -  g(x-y) \,f(x) + f(x-y)\, g(x-y) = \\
& \qquad \qquad = -(f(x-y) - f(x))\,g(x) 
-(g(x-y) - g(x))\,f(x) 
+(f(x-y)\, g(x-y) - f(x)g(x)).
\end{align*}
For the third one, we compute
$$
\langle B_{\eps}[u](\cdot),\varphi(\cdot)\rangle_{L^{2}(\Td)} = \int_{\Td} \int_{\Td} \frac{\omega_{\varepsilon}(y)}{\varepsilon^2} (u(x) - u(x-y)) \, \varphi(x) \diff y \diff x.
$$
Changing variables $x' = x-y$, $y' = -y$ and using symmetry of the kernel
$$
\langle B_{\eps}[u](\cdot),\varphi(\cdot)\rangle_{L^{2}(\Td)} = \int_{\Td} \int_{\Td} \frac{\omega_{\varepsilon}(y)}{\varepsilon^2} (u(x'-y') - u(x')) \, \varphi(x'-y') \diff y' \diff x'.
$$
Therefore,
\begin{align*}
2\, \langle B_{\eps}[u](\cdot),\varphi(\cdot)\rangle_{L^{2}(\Td)} &= 
\int_{\Td} \int_{\Td} \frac{\omega_{\varepsilon}(y)}{\varepsilon^2} (u(x) - u(x-y)) \, (\varphi(x) - \varphi(x-y)) \diff y \diff x \\
&= \int_{\Td} \int_{\Td} \frac{\sqrt{\omega_{\varepsilon}(y)}}{\varepsilon} (u(x) - u(x-y)) \, \frac{\sqrt{\omega_{\varepsilon}(y)}}{\varepsilon}  (\varphi(x) - \varphi(x-y)) \diff y \diff x.
\end{align*}
Finally, to prove \ref{propS:conv_to_energy} we use the definition of $\omega_{\eps}$ and change variables with respect to $y$:
$$
\int_0^T \int_{\Td} \int_{\Td} (S_{\varepsilon}[u_{\varepsilon}])^2 \, \varphi(t,x) = 
 \int_{B_1}  \omega(y) \int_0^T \int_{\Td} \varphi(t,x) \frac{|u_{\varepsilon}(t,x) - u_{\varepsilon}(t,x-\varepsilon y)|^2}{2\varepsilon^2} \diff x \diff t \diff y
$$
For fixed $y$,
$$
\int_0^T \int_{\Td} \varphi(t,x) \frac{|u_{\varepsilon}(t,x) - u_{\varepsilon}(t,x-\varepsilon y)|^2}{\varepsilon^2} \diff x \diff t \to 
\int_0^T \int_{\Td}\varphi(t,x) |\nabla u(x)|^2 |y|^2   \diff x \diff t,
$$
$$
\left|\int_0^T \int_{\Td} \varphi(t,x) \frac{|u_{\varepsilon}(t,x) - u_{\varepsilon}(t,x-\varepsilon y)|^2}{\varepsilon^2} \diff x \diff t\right| \leq \| \varphi\|_{\infty} \, \sup_{\varepsilon} \|Du_{\varepsilon}\|_{2}^2 \, |y|^2
$$
due to Lemma \ref{lem:diff_quot_strong_conv}. As the majorant is integrable, the dominated convergence theorem concludes the proof.
\end{proof}

Since $B_{\eps}$ has a similar behavior as the Laplace operator, one can expect that $S_{\eps}$ acts like a gradient (in $L^2(\Td)$). Nevertheless, note that $S_{\eps}[\varphi](x,y)$ is a scalar. From now on, we write $\nabla S_{\eps}$ for the gradient of $S_{\eps}$ with respect to the variable $x$ \ie 
$$
\nabla S_{\eps}[\varphi](x,y):=\f{\sqrt{\omega_{\eps}(y)}}{\sqrt{2}\eps}(\nabla\varphi(x-y)-\nabla\varphi(x)).
$$

\begin{proof}[Proof of Theorem \ref{thm:final}] We only have to pass to the limit in the term $\int_{0}^{T}\int_{\Td}\DIV(u_{\eps}\nabla\mu_{\eps})\varphi\diff x\diff t$ where $\varphi \in C^3([0,T]\times \Td)$. Integrating by parts, we obtain 
\begin{equation}\label{eq:split_for_I1_I2_I3}
    \begin{split}
\int_{0}^{T}\int_{\Td}\DIV(u_{\eps}\nabla\mu_{\eps})\,&\varphi\diff x\diff t = -\int_{0}^{T}\int_{\Td}u_{\eps}\nabla\mu_{\eps}\cdot\nabla\varphi\diff x\diff t =\int_{0}^{T}\int_{\Td}B_{\eps}[u_{\eps}]\nabla u_{\eps}\cdot\nabla\varphi\diff x\diff t\\
&+\int_{0}^{T}\int_{\Td}B_{\eps}[u_{\eps}]u_{\eps}\Delta\varphi\diff x\diff t -\int_{0}^{T}\int_{\Td}u_{\eps}F''(u_{\eps})\nabla u_{\eps}\cdot\nabla\varphi\diff x\diff t\\
&=:I_{1}+I_{2}+I_{3}. \phantom{\int_{0}^{T}\int_{\Td}}
\end{split}
\end{equation}

\underline{\textit{Step 1: Compactness}}. Using Lemma \ref{lem:uniform_est_just_eps} and Lemma \ref{lem:compactness-FK-tx} we can choose a subsequence of $\{\ueps\}_{\varepsilon}$ such that 
\begin{enumerate}[label=(D\arabic*)]
    \item $\partial_t \ueps \rightharpoonup \partial_t u$ weakly in $L^2(0,T; W^{-1,s'}(\Td))$,
    \item $\ueps \to u$ strongly in $L^2((0,T)\times \Td)$,
    \item $\nabla \ueps \to \nabla u$ strongly in $L^2((0,T)\times \Td)$,
    \item $\sqrt{F_{1}''(u_\eps)}\nabla u_\eps \rightharpoonup \xi$ weakly in $L^2((0,T)\times \Td)$.
\end{enumerate}
\underline{\textit{Step 1: Convergence of $I_{1}$}}. Using~\ref{propS_nonneg} in Lemma \ref{lem:S_properties} we write
\begin{align*}
I_{1}&=\int_{0}^{T}\int_{\Td}\int_{\Td}S_{\eps}[u_{\eps}]S_{\eps}(\nabla u_{\eps}\cdot\nabla\varphi)\diff x\diff y\diff t\\
&=\int_{0}^{T}\int_{\Td}\int_{\Td}S_{\eps}[u_{\eps}]S_{\eps}(\nabla u_{\eps})\cdot\nabla\varphi\diff x\diff y\diff t+\int_{0}^{T}\int_{\Td}\int_{\Td}S_{\eps}[u_{\eps}]\nabla u_{\eps}\cdot S_{\eps}[\nabla\varphi]\diff x\diff y\diff t+R_{\eps}\\
&=J_{1}^{(1)}+J_{2}^{(1)}+R_{\eps}^{(1)},
\end{align*}
where $R_{\eps}^{(1)}$ is defined as 
\begin{equation*}
 R_{\eps}^{(1)}= \int_{0}^{T}\int_{\Td}\int_{\Td}S_{\eps}[u_{\eps}]\left(S_{\eps}(\nabla u_{\eps}\cdot\nabla\varphi)-S_{\eps}(\nabla u_{\eps})\cdot\nabla\varphi-\nabla u_{\eps}\cdot S_{\eps}[\nabla\varphi]\right)\diff x\diff y \diff t.
\end{equation*}
For $J_{1}^{(1)}$ we use identity
$$
S_{\eps}[u_{\eps}] \, S_{\eps}(\nabla u_{\eps}) =
 S_{\eps}[u_{\eps}] \, \nabla S_{\eps}(u_{\eps}) =
\frac{1}{2} \nabla \left|S_{\eps}[u_{\eps}]\right|^2
$$
so after integration by parts we obtain
\begin{equation}\label{eq:convJ1}
J_{1}^{(1)}=-\f{1}{2} \int_{0}^{T}\int_{\Td}\int_{\Td}(S_{\eps}[u_{\eps}])^{2}\Delta\varphi\diff x\diff y\diff t  \to -\f{D}{2}\int_{0}^{T}\int_{\Td}|\nabla u|^{2}\Delta\varphi\diff x\diff t 
\end{equation}
due to \ref{propS:conv_to_energy} in Lemma \ref{lem:S_properties}. For $J_2^{(1)}$ we change variables to have
\begin{align*}
J_2^{(1)} &=\f 1 2 \int_0^T \int_{\Td}\int_{\Td} {\omega_\eps}(y)\, \frac{u_{\eps}(x-y) - u_{\eps}(x)}{\varepsilon} \nabla u_{\eps}(x) \cdot \frac{\nabla\varphi(x-y) - \nabla\varphi(x)}{\eps} \diff x\diff y \diff t = \\
& =\f 1 2 \int_{\Td}{\omega}(y)  \int_0^T \int_{\Td}  \frac{u_{\eps}(x-\eps y) - u_{\eps}(x)}{\varepsilon} \nabla u_{\eps}(x) \cdot \frac{\nabla\varphi(x-\eps y) - \nabla\varphi(x)}{\eps} \diff x \diff t\diff y.
\end{align*}
We are first concerned with the inner integral. With Lemma~\ref{lem:diff_quot_strong_conv} we have that for fixed $y \in \Td$
$$
\frac{u_{\eps}(x-\eps y) - u_{\eps}(x)}{\varepsilon} \to - \nabla u(x) \cdot y \mbox{ in } L^2((0,T)\times \Td).
$$
Moreover, a Taylor expansion implies that
$$
\frac{\nabla\varphi(x-\eps y) - \nabla\varphi(x)}{\varepsilon} \to - D^{2}\varphi(x) y \mbox{ in } L^\infty((0,T)\times \Td;\R^{d}).
$$
Combining this with a strong convergence $\nabla u_{\varepsilon} \to \nabla u$ in $L^2((0,T)\times \Td)$, we deduce
\begin{multline*}
\int_0^T \int_{\Td}  \frac{u_{\eps}(x-\eps y) - u_{\eps}(x)}{\varepsilon} \nabla u_{\eps}(x) \cdot \frac{\nabla\varphi(x-\eps y) - \nabla\varphi(x)}{\eps} \diff x \diff t \to \\
\to \int_0^T \int_{\Td}
\nabla u(x) \cdot y \, \nabla u(x) \cdot (D^2 \varphi(x) y) \diff x \diff t.
\end{multline*}
Finally, we apply the dominated convergence theorem to the integral with respect to $y$ with the dominating function $\norm{D^{2}\varphi}_{\infty}\sup_{\eps}\norm{\nabla u_{\eps}}_{2}^{2} |y|^{2}$. We obtain
\begin{equation}\label{eq:convJ2}
\begin{split}
J_2^{(1)} \to \f 1 2 \int_{\Td} \omega(y) |y|^2 \diff y \int_0^T \int_{\Td} \nabla u(x) &\cdot D^2 \varphi(x) \nabla u(x) \diff x \diff t= \\
&=D\int_0^T \int_{\Td}(\nabla u(x)\otimes \nabla u(x)) : D^2 \varphi(x) \diff x \diff t,
\end{split}
\end{equation}
where we also used the symmetry of $D^{2}\varphi$ and properties of $\omega$ defined in~\eqref{as:omega}. It remains to deal with the error term. Using \ref{propS_product_rule} in Lemma \ref{lem:S_properties} we can write 
\begin{equation*}
R_{\eps}^{(1)}=\int_{0}^{T}\int_{\Td}\int_{\Td}S_{\eps}[u_\eps]\f{\sqrt{\omega_{\eps}(y)}}{\sqrt{2}\eps}[(\nabla u_{\eps}(x-y)-\nabla u_{\eps}(x))\cdot(\nabla\varphi(x-y)-\nabla\varphi(x))]\diff x\diff y\diff t.    
\end{equation*}
We want to prove that $R_{\eps}^{(1)}$ converges to 0. By Cauchy-Schwarz inequality (in time and space) as well as bounds on $S_{\eps}[u_{\eps}]$ it remains to prove that
\begin{equation}\label{eq:convReps}
\int_{0}^{T}\int_{\Td} \int_{\Td}\f{\omega_{\eps}(y)}{\eps^2}|\nabla u_{\eps}(x-y)-\nabla u_{\eps}(x)|^{2}|\nabla\varphi(x-y)-\nabla\varphi(x)|^{2}\diff y \diff x\diff t\to 0.     
\end{equation}
Using Taylor's expansion we can estimate this integral with
\begin{equation*}
\eps\norm{D^{2}\varphi}_{L^{\infty}}\left(\int_{0}^{T}\int_{\Td}\int_{\Td}\f{\omega_{\eps}(y)}{\eps^2}|\nabla u_{\eps}(x-y)-\nabla u_{\eps}(x)|^{2} \diff y\diff x\diff t\right) 
\end{equation*}
which converges to zero by the bound from the entropy \eqref{eq:entropy} so that \eqref{eq:convReps} follows. We conclude that
\begin{equation*}
I_{1}\to -\f{D}{2}\int_{0}^{T}\int_{\Td}|\nabla u|^{2}\Delta\varphi\diff x\diff t+D\int_0^T \int_{\Td}(\nabla u\otimes \nabla u) : D^2 \varphi \diff x \diff t.  
\end{equation*}

\underline{\textit{Step 2: Convergence of $I_{2}$}}. We observe that the only differences between $I_1$ and $I_2$ are $u_{\varepsilon}$ and $\Delta \varphi$ in place of $\nabla u_{\varepsilon}$ and $\nabla \varphi$ respectively. As we have the same (in fact, better) estimates for these quantities, the proof is the same and we conclude 
\begin{equation*}
I_{2}\to D \int_{0}^{T}\int_{\Td}|\nabla u|^{2}\Delta\varphi\diff x\diff t+D\int_0^T \int_{\Td}u\, \nabla u\cdot \nabla\Delta\varphi.  
\end{equation*}

\underline{\textit{Step 3: Convergence of $I_{3}$}}. For $I_{3}$ the proof is similar to the reasoning in Steps 1, 3 and 6 of the proof of Theorem~\ref{thm:weaksoldelta}because we have to use the same estimates. Roughly speaking, one proves that $\ueps \to u$ strongly in $L^{k}((0,T)\times \Td)$ by interpolation so that one can identify $\xi = \sqrt{F_1''(u)} \nabla u$. Next, convergence in $L^{k}((0,T)\times \Td)$ allows also to prove strong convergence $u_{\varepsilon} \sqrt{F_1''(u_{\varepsilon})} \to u \sqrt{F_1''(u)}$ in $L^2((0,T)\times\Td)$ thanks to growth condition \ref{assumption_pot_3} while the convergence $u_{\varepsilon} \sqrt{F_2''(u_{\varepsilon})} \to u \sqrt{F_2''(u)}$ in $L^2((0,T)\times\Td)$ is trivial because $F_2''\in L^{\infty}$. This shows that 
\begin{equation*}
I_{3}\to -\int_{0}^{T}\int_{\Td}u F''(u)\nabla u\cdot\nabla\varphi\diff x\diff t.
\end{equation*}
\underline{\textit{Conclusion of Steps 1-3.}} In the limit $\varepsilon \to 0$ we obtain
\begin{multline*}
\int_{0}^{T}\langle\p_{t}u,\varphi\rangle_{(W^{-1,s'}(\Td),W^{1,s}(\Td))} = D\int_0^T \int_{\Td}(\nabla u\otimes \nabla u) : D^2 \varphi +
\\
+  \f{D}{2}\int_{0}^{T}\int_{\Td}|\nabla u|^{2}\Delta\varphi
+D\int_0^T \int_{\Td}u\,\nabla u\cdot \nabla\Delta\varphi -\int_{0}^{T}\int_{\Td}uF''(u)\nabla u\cdot\nabla\varphi.     
\end{multline*}
\underline{\textit{Step 4: Regularity of $u$ and better weak formulation}}
Now we  prove the regularity of the limit function $u$. This allows us to perform integration by parts on the different terms using the formula~\eqref{integration_by_parts} and recover the Definition~\ref{def:weak_sol_limit}. In fact, in the limit $\varepsilon \to 0$, from the entropy we obtain (see~\cite[Theorem 4]{bourgain2001another} and~\cite[Theorem 1.2]{MR2041005}) 
	$$
	\sum_{i,j=1}^d \int_0^t \int_{\Td} |\partial_{x_i} \partial_{x_j} u|^2\le \liminf_{\eps\to 0}	\frac{1}{4\eps^{2}}\int_{0}^{t}\int_{\Td} \int_{\Td} \omega_{\varepsilon}(y) \, |\nabla u_\eps(x) - \nabla u_\eps(x-y)|^2  
	$$
so in the limit $\varepsilon \to 0$ we gain one more derivative. Then, since
$$
D\int_0^T \int_{\Td}u \nabla u \cdot \nabla\Delta\varphi = - D \int_0^T \int_{\Td} \Delta \varphi \, |\nabla u|^2 - D \int_0^T \int_{\Td} u\,\Delta u\, \Delta \varphi
$$
and using formula ~\eqref{integration_by_parts}, we compute
\begin{align*}
I_1 + I_2 &= 
D\int_0^T \int_{\Td}(\nabla u\otimes \nabla u) : D^2 \varphi \, - \f{D}{2}\int_{0}^{T}\int_{\Td}|\nabla u|^{2}\Delta\varphi - D\int_{0}^{T}\int_{\Td} u\, \Delta u\,\Delta\varphi\\
&= -D\int_{0}^{T}\int_{\Td} \Delta u\, \nabla u \cdot \nabla \varphi -D\int_{0}^{T}\int_{\Td} u\, \Delta u\,\Delta\varphi.
\end{align*}
This ends the proof of Theorem~\ref{thm:final}.
\end{proof}

\appendix
\section{Results from classical analysis}
\subsection{Difference quotients}
\begin{lem}\label{lem:diff_quot_strong_conv}
Let $\{\ueps\}$ be a sequence strongly compact in $L^2(0,T; H^1(\Td))$. Then, for fixed $y \in \Td$,
$$
\frac{u_{\eps}(t,x-\eps y) - u_{\eps}(t,x)}{\varepsilon} \to - \nabla u(t,x) \cdot y \mbox{ strongly in } L^2((0,T)\times \Td).
$$
\end{lem}
\begin{proof}
We write 
\begin{align*}
\frac{u_{\eps}(t,x-\eps y) - u_{\eps}(t,x)}{\varepsilon}&=-y\cdot\int_{0}^{1}\nabla u_{\eps}(t,x-\eps\theta y) \diff\theta \\
&=-y\cdot\int_{0}^{1} \left(\nabla u_{\eps}(t,x-\eps\theta y)-\nabla u_{\eps}(t,x)\right)\diff\theta - y\cdot \nabla u_{\eps}(t,x).
\end{align*}
By assumption $y\cdot \nabla u_{\eps}\to y\cdot \nabla u$ strongly in $L^{2}((0,T)\times \Td)$ so we only have to prove that the first term on the (RHS) converges to 0. By Fubini's theorem and Cauchy-Schwarz inequality
\begin{align*}
\int_{0}^{T}&\int_{\Td}\Big|\int_{0}^{1}\left(\nabla u_{\eps}(t,x-\eps\theta y)-\nabla u_{\eps}(t,x)\right)\diff\theta\Big|^{2}\diff x\diff t\\ 
&\le C \int_{0}^{1}\int_{0}^{T}\int_{\Td}|\nabla u_{\eps}(t,x-\eps\theta y)-\nabla u_{\eps}(t,x)|^{2}\diff x\diff t\diff \theta =C\int_{0}^{1}\norm{\tau_{\eps\theta y}\nabla u_{\eps}-\nabla u_{\eps}}_{L^{2}((0,T)\times \Td)}^{2}\diff \theta,
\end{align*}
where $\tau$ is the translation operator. The last term converges to 0 when $\eps\to 0$ by the Fréchet Kolmogorov theorem. 
\end{proof}

\subsection{Growth estimates on mollified nonlinearity}
\begin{lem}\label{lem:F_mollification_estimates}
Let $F$ satisfies Assumption \ref{ass:potentialF} with constants $C_{1}$, ..., $C_{10}$. Then, $F_{\delta} = F \ast \eta_{\delta}$ with $0 \leq \delta \leq 1$ satisfies Assumption \ref{ass:potentialF} with constants
\begin{align*}
&\widetilde{C_{1}} = 2^{1-k} C_{1},  \qquad  &&\widetilde{C_{2}} = C_1 + C_2,
\qquad &&\widetilde{C_{3}} = 2^{k-1}\,{C_{3}},  \qquad  &&\widetilde{C_{4}} = \widetilde{C_{3}} +  C_{4},\\
&\widetilde{C_{5}} = \min(2^{3-k},1) \, C_{5}, \qquad &&\widetilde{C_{6}} = C_{5} + C_6,
\qquad 
&&\widetilde{C_{7}} = \max(2^{k-3},1)\, C_{7}, \qquad &&\widetilde{C_{8}} = \widetilde{C_{7}} + C_{8}.\\
&\widetilde{C_{9}} = C_9 + 2\,C_{10}, \qquad &&\widetilde{C_{10}} = 2\,C_{10}. && &&
\end{align*}
\end{lem}
\begin{proof}
We decompose $F_{\delta,1} = F_1 \ast \eta_{\delta}$ and $F_{\delta,2} = F_2 \ast \eta_{\delta}$. Suppose that $F_1(u) \leq C_{3}|u|^{k} + C_{4}$. Then,
$$
F_{\delta,1}(u) = \int_{\R} F_1(u-s)\, \eta_{\delta}(s) \diff s \leq 
C_{3} \, \int_{\R} |u-s|^{k} \eta_{\delta}(s) \diff s + C_{4} \leq 
2^{k-1} C_{3} |u|^{k} + 2^{k-1} C_{3} + C_{4}
$$
where we used inequality valid for $p \geq 0$
\begin{equation}\label{eq:holder_bound_u-s}
|u-s|^{p} \leq \max(1,2^{p-1})\,(|u|^{p} + |s|^{p}).
\end{equation}
It follows that $\widetilde{C_{3}} = 2^{k-1} C_{3}$ and $\widetilde{C_{4}} = 2^{k-1} C_{3} + C_{4}$. In a similar way, we compute constants $\widetilde{C_{7}}$, $\widetilde{C_{8}}$. For $\widetilde{C_{1}}, \widetilde{C_{2}}$, $\widetilde{C_{5}}, \widetilde{C_{6}}$ the reasoning is the same but we have to use a lower bound of the form
$$
|u - s|^p  \geq \min(1, 2^{1-p}) \, |u|^p - |s|^p.
$$
so that, for example, if $F_1 \geq C_1 |u|^k - C_2$ we have
$$
F_{\delta,1}(u) = \int_{\R} F_1(u-s)\, \eta_{\delta}(s) \diff s \geq 
C_{1} \int_{\R} |u-s|^k \eta_{\delta}(s) \diff s - C_{2} \geq 
2^{1-k} \, C_{1} \, |u|^p - C_{1} - C_2.
$$
For the constants $\widetilde{C_9}, \widetilde{C_{10}}$ we argue using \eqref{eq:holder_bound_u-s} once again
$$
F_{\delta,2}(u) 
\geq - C_{9} -C_{10} \, \int_{\R} |u-s|^{2} \eta_{\delta}(s) \diff s
\geq -C_9 - 2\,C_{10} - 2\,C_{10}\,|u|^2.
$$
\end{proof}

\subsection{Potentials satisfying Assumption \ref{ass:potentialF}}
\begin{lem}\label{lem:potentials_satisfying_ass}
Let $F$ be as in \ref{ex_general_F} in Example \ref{ex:potentials}. Then, $F$ satisfies Assumption \ref{ass:potentialF}.
\end{lem}
\begin{proof}
On $\R \setminus I$ we define $F_1(u) = F(u)$. By \cite[Theorem 3.2]{yan2012extension}, there exists a $C^2$ extension of $F_1$ to $\R$ denoted by $F_1$ which preserves convexity, i.e. $F_1''(u)>b>0$ for some $b>0$. Moreover, $F_1$ has $k$-growth on $\R$ (in fact, by continuity, the behaviour of $F_1$ on $I$ can be included in constants $C_{2}$, $C_4$, $C_6$, $C_{8}$ in Assumption \ref{ass:potentialF}). We finally define
$$
F_2 = \begin{cases}
F(u) - F_1(u) &\mbox{ on } I,\\
0 &\mbox{ on } \R \setminus I.
\end{cases}
$$
Function $F_2$ is $C^2$ because at the endpoints of interval $I$ we have $F'' = F_1''$ as $F_1$ is $C^2$ extension of $F$. Finally, $F_2$ satisfies condition \ref{assumption_pot_4} in Assumption \ref{ass:potentialF} with $F_2(u) \geq - \|F_2\|_{\infty}$.
\end{proof}

\subsection{Proof of Lemma \ref{lem:properties_phi_delta}}\label{subsect:proof_phi_delta}
\begin{proof}
First, we note the formula which will be useful
$$
\phi(x) = \int_1^x \int_1^y \frac{1}{z} \diff z \diff y.
$$
Now, we proceed to the proof. First, (1) follows from the definition. Next, (2) 
follows from writing 
\begin{equation}\label{eq:formula_for_phi_delta}
\phi_{\delta}(x) = \int_\R \int_\R \frac{1}{T_{\delta}(z)} \mbox{sgn}(y-1) \, \mbox{sgn}(x-1) \, \mathds{1}_{y \in [1,x]} \, \mathds{1}_{z \in [1,y]} \diff z \diff y,
\end{equation}
and dominated convergence (for fixed $x>0$). Then, (3) follows from $T_{\delta} \geq 0$ and the observation that $x \geq 1$, $x<1$ implies $y\geq 1$, $y<1$ respectively.\\

To see (4), we distinguish three cases.
\begin{itemize}
    \item When $x \geq \frac{1}{\delta}-1$, we split the integrals and use the estimate $T_{\delta}(x) \geq \frac{1}{\delta}-1$ so that
\begin{align*}
\phi_{\delta}(x) \leq \int_1^{\frac{1}{\delta}-1} \int_1^y \frac{1}{z} \diff z& \diff y +
\int_{\frac{1}{\delta}-1}^x \int_1^y \frac{1}{\frac{1}{\delta}-1} \diff z \diff y\leq 
\\ &\leq \phi\left(\frac{1}{\delta}-1 \right) +  \frac{\delta}{2(\delta-1)} x^2 \leq 
\phi(x) +  \frac{\delta}{\delta-1} (x-1)^2,
\end{align*}
because $\phi(x)$ is non-decreasing for $x \geq 1$.
\item When, $x \in \left(2\delta, \frac{1}{\delta}-1 \right)$ we have $\phi_{\delta} = \phi$ because on this set $T_{\delta}(z) = z$. 
\item When $x \in [0,2\delta]$ we have a lower bound $T_{\delta}(x) \geq \delta$ so that 
$$
\phi_{\delta}(x) \leq \int_{x}^{2\delta} \int_{y}^{1} \frac{1}{\delta} + \int_{2\delta}^1 \int_{y}^{1} \frac{1}{z} \diff z \leq 2 + \phi(2\delta) \leq 3 
$$
as $\phi(2\delta) \leq \phi(0) = 1$ because $\phi(x)$ is decreasing for $x \in (0,1)$. 
\end{itemize} 

Finally, to see (5), let $x < 0$. Then,
$$
\phi_{\delta}(x) \geq  \int_x^0 \int_y^0 \frac{1}{\delta} \diff z \diff y = \frac{1}{\delta} \int_x^0 -y \diff y = \frac{x^2}{2\delta}.
$$
\end{proof}

\section{Bourgain-Brézis-Mironescu and Ponce compactness result}
We upgrade here the result of \cite[Proposition 4.2]{MR2041005} and \cite[Theorem 4]{bourgain2001another} to the time-space setting. We consider sequence of radial functions $\{\rho_{\eps}\}$ such that $\rho_{\eps} \geq 0$, $\int_{\R^d} \rho_{\eps} = 1$ and 
$$
\lim_{\eps \to 0} \int_{|x| > \delta} \rho_{\eps}(x) \diff x = 0 \mbox{ for all } \delta > 0.
$$
For the formulation of the compactness result, we use another sequence $\{\varphi_{\delta}\}_{\delta \in (0,1)} \subset C_{c}^{\infty}(\R^{d})$ of standard mollifiers with mass~1 such that $\varphi_{\delta}(x)=\frac{1}{\delta^{d}}\varphi(\frac{x}{\delta})$ with $\varphi$ of mass 1 and compactly supported.

\begin{thm}\label{thm:ponce_tx}
Let $d \geq 2$. Let $\{f_\eps\}$ be a sequence bounded in $L^p((0,T)\times \Td)$. Suppose that there exists a sequence $\{\rho_{\eps}\}$ as above such that
\begin{equation}\label{eq:Ponce_org_condition}
\int_0^T \int_{\Td} \int_{\Td} \frac{|f_\eps(t,x) - f_\eps(t,y)|^p}{|x-y|^p} \rho_\eps(|x-y|) \diff x \diff y \diff t \leq C
\end{equation}
for some constant $C$. Then, $\{f_\eps\}$ is compact in space in $L^p((0,T)\times \Td)$, i.e.
\begin{equation}\label{eq:equicontinuity_Lp_space}
\lim_{\delta \to 0} \limsup_{\varepsilon \to 0} \int_0^T \int_{\Td} |f_\eps \ast \varphi_{\delta}(t,x) - f_\eps(t,x)|^p \diff x \diff t = 0.
\end{equation}
\end{thm}
\begin{rem}\label{rem:adapt_ponce}
Let $\omega:\R^d \to \R$ be a smooth function, supported in the unit ball such that $\int_{\R^d} \omega(x) \diff x = 1$. Consider $\omega_{\eps} = \frac{1}{\eps^d} \omega\left(\frac{x}{\eps}\right)$. Suppose that
$$
\int_0^T \int_{\Td} \int_{\Td} \frac{|f_\eps(x) - f_\eps(y)|^p}{\eps^p} \omega_\eps(|x-y|) \diff x \diff y \diff t \leq \widetilde{C}.
$$
Then, \eqref{eq:Ponce_org_condition} is satisfied. Indeed, we consider
\begin{equation}\label{eq:new_rescaled_kernel}
\rho_{\eps}(x) = \frac{\omega_\eps(|x|)\, |x|^p}{\eps^p \, \int_{\R^d} \omega(y) |y|^p \diff y}
\end{equation}
so that \eqref{eq:Ponce_org_condition} holds true with $\frac{\widetilde{C}}{\int_{\R^d} \omega(y) |y|^p \diff y}$.
\end{rem}

\begin{proof}[Proof of Theorem \ref{thm:ponce_tx}] The result for sequences that do not depend on time has been obtained in \cite{bourgain2001another, MR2041005}. To demonstrate that it is sufficient to {\it integrate in time} the reasoning mentioned above, we make an additional assumption that for every $\eps$, $\rho_{\eps}$ is a nonincreasing function as in in~\cite[Theorem 4]{bourgain2001another}. For the general case, one has to proceed as in~\cite[Theorem 1.2]{MR2041005}.\\

\noindent We define
\begin{align*}
F_{\varepsilon}(s) &:= \int_0^T \int_{|y|=1} \int_{\Td} |f_{\varepsilon}(t,x+sy) - f_{\varepsilon}(t,x)|^p \diff x \diff y \diff t\\
 &= \frac{1}{s^{d-1}} \int_0^T \int_{|y|=s} \int_{\Td} |f_{\varepsilon}(t,x+y) - f_{\varepsilon}(t,x)|^p \diff x \diff y \diff t.
\end{align*}
By virtue of the computation above, we can express the assumption \ref{eq:Ponce_org_condition} using function $F_{\varepsilon}$ as follows
\begin{equation}\label{eq:ass_in_terms_of_Fe}
 \int_0^{\delta} s^{d-1} \frac{F_{\varepsilon}(s)\,\rho_{\varepsilon}(s)}{s^p} \diff s \leq C.
\end{equation}
Using the triangle inequality
$$
|f_{\varepsilon}(t,x+2sy) - f_{\varepsilon}(t,x)| \leq |f_{\varepsilon}(t,x+2sy) - f_{\varepsilon}(t,x + sy)| + |f_{\varepsilon}(t,x+sy) - f_{\varepsilon}(t,x)|
$$
and change of variables we obtain
\begin{equation}\label{eq:doubling_cond_F}
F_{\varepsilon}(2s) \leq 2^p F_{\varepsilon}(s), \qquad \frac{F_{\varepsilon}(2s)}{(2s)^p} \leq \frac{F_{\varepsilon}(s)}{s^p}.
\end{equation}
We estimate by Jensen's inequality
\begin{equation}\label{eq:estimate_conv_minus_f}
\begin{split}
\int_0^T \int_{\Td} &|f_{\varepsilon} \ast \varphi_{\delta} - f_{\varepsilon}|^p \diff x \diff t \leq \frac{C}{\delta^d} \int_0^T \int_{\Td} \int_{|x-y|\leq \delta} 
|f_{\varepsilon}(x) - f_{\varepsilon}(y)|^p
\diff y \diff x \diff t\\
&=\frac{C}{\delta^d} \int_0^T \int_{\Td} \int_{|h|\leq \delta} 
|f_{\varepsilon}(x+h) - f_{\varepsilon}(x)|^p
\diff h \diff x \diff t\\
&=\frac{C}{\delta^d} \int_0^T \int_{\Td}
\int_{0}^{\delta} s^{d-1} \int_{|h| = s}
|f_{\varepsilon}(x+h) - f_{\varepsilon}(x)|^p
\diff h \diff s \diff x \diff t = \frac{C}{\delta^d} \int_{0}^{\delta} s^{d-1} F_{\varepsilon}(s) \diff s.
\end{split}
\end{equation}
Now, we use functional inequality (which requires doubling condition \eqref{eq:doubling_cond_F}, cf. \cite[Eq. (24)]{bourgain2001another}) 
\begin{equation}\label{eq:funct_ineq_markov}
\delta^{-d} \int_0^{\delta} s^{d-1} \frac{F_{\varepsilon}(s)}{s^p} \diff s \leq C(d)\, \frac{ \int_0^{\delta} s^{d-1} \frac{F_{\varepsilon}(s)\,\rho_{\varepsilon}(s)}{s^p} \diff s}{\int_{|x| < \delta} \rho_{\varepsilon}(x) \diff x}
\end{equation}
For each $\delta>0$, there exists $\varepsilon(\delta)$ such that for all $\varepsilon < \varepsilon(\delta)$ we have
$\int_{|x| < \delta} \rho_{\varepsilon}(x) \diff x = 1$. In particular, for $\varepsilon < \varepsilon(\delta)$ we have by \eqref{eq:funct_ineq_markov} and \eqref{eq:ass_in_terms_of_Fe}
$$
\delta^{-d} \int_0^{\delta} s^{d-1} \frac{F_{\varepsilon}(s)}{s^p} \diff s \leq  C(d)\, \delta^{p}.
$$
In view of \eqref{eq:estimate_conv_minus_f}, the proof is concluded.

\end{proof}

\section{nonlocal Poincaré inequalities}

Let $\omega:\R^d \to \R$ be a smooth function, supported in the unit ball such that $\int_{\R^d} \omega(x) \diff x = 1$. Consider $\omega_{\eps} = \frac{1}{\eps^d} \omega\left(\frac{x}{\eps}\right)$.

\begin{lem}\label{lem:Poincare_with_average}
There exists $C_{p}$ and $\varepsilon_0^A$ such that 
\begin{equation*}
\int_{\Td}|f-(f)_{\Td}|^{2}\le \f{1}{4C_{p}}\int_{\Td} \int_{\Td} \frac{|f(t,x) - f(t,y)|^2}{\eps^2} \omega_\eps(|x-y|) \diff x \diff y     
\end{equation*}
for every $f\in L^{2}(\Td)$ and $\eps\le\eps_{0}^A$. 
\end{lem}

For the proof, we refer to Ponce~\cite[Theorem 1.1]{MR2041005} with kernel given by \eqref{eq:new_rescaled_kernel}. We also have an opposite inequality from \cite[Theorem 1]{bourgain2001another}:

\begin{lem}\label{lem:inv_poincare_ineq} For all $f \in H^1(\Td)$
$$
\int_{\Td} \int_{\Td} \frac{|f(x)- f(y)|^2}{\varepsilon^2} \omega_{\varepsilon}(x-y) \diff x \diff y \leq C( \Td) \, \|f\|_{H^1(\Td)}^2.
$$
\end{lem}

Finally, we formulate a variant of Lemma \ref{lem:Poincare_with_average} which does not require an average on the left-hand side.

\begin{lem}\label{lem:poincare_nonlocal_H1_L2}
For each $\gamma \in (0,1)$ there exists ${\varepsilon}_0^B$ and constant $C(\gamma)$ such that for all $\varepsilon \in (0, {\varepsilon}_0^B)$ and all $f \in H^1(\Td)$ we have
$$
\| f\|^2_{H^1(\Td)} \leq \gamma \int_{\Td} \int_{\Td}  \frac{|\nabla f(x) - \nabla f(y)|^2}{\eps^2} \omega_\eps(|x-y|) \diff x \diff y + C(\gamma) \|f \|^2_{L^2(\Td)}.
$$
\end{lem}
\begin{proof}
Aiming at a contradiction, suppose that there exists $\gamma$ with the following property: there exists sequence $\{\varepsilon_n\}$ with $0<\varepsilon_n < \frac{1}{n}$ and sequence $\{f_n\}$ such that
$$
\| f_{n}\|^2_{H^1(\Td)} > \gamma \int_{\Td} \int_{\Td}  \frac{|\nabla f_{n}(x) - \nabla f_{
n}(y)|^2}{\eps_n^2} \omega_{\eps_n}(|x-y|) \diff x \diff y + n\, \|f_{n}\|^2_{L^2(\Td)}.
$$
As $\| f_{n}\|_{H^1(\Td)} > 0$, we may define $g_{n} := \frac{f_{n}}{\| f_{n}\|_{H^1(\Td)}}$. Note that $\|g_{n}\|_{H^1(\Td)} = 1$ and
$$
1 > \gamma \int_{\Td} \int_{\Td}  \frac{|\nabla g_{n}(x) - \nabla g_{n}(y)|^2}{\eps_n^2} \omega_{\eps_n}(|x-y|) \diff x \diff y + n\, \|g_{n}\|_{L^2(\Td)}^2.
$$
The first term gives compactness of the gradients (because $\{g_{n}\}$ is bounded in $H^1(\Td)$ so that, together with Rellich-Kondrachov, there exists function $g$ such that $g_{n} \to g$ in $H^1(\Td)$ (after passing to a subsequence). But then $g = 0$ because $n\, \|g_{n}\|_{L^2(\Td)} < 1$. This is however contradiction with $\|g\|_{H^1(\Td)} = \lim_{n \to \infty} \|g_n\|_{H^1(\Td)} = 1$.
\end{proof}

\section{Compactness in time/space with the Fréchet-Kolmogorov theorem}\label{app:FK_tx}

\begin{lem}
Suppose that $\{f_{\varepsilon}\}$ is a sequence bounded in $L^2((0,T)\times \Td)$ such that
\begin{itemize}
    \item $\p_{t}f_{\eps}= \nabla^k (J_{\eps})$, where $\nabla^k$ is any linear differential operator of order $k \in \N$ and $\{J_{\eps}\}$ uniformly bounded in $L^{1}((0,T)\times \Td)$,
    \item  $\{f_\eps\}$ is compact in space in $L^2((0,T)\times \Td)$, i.e.
\begin{equation}\label{item2compactness_2}
\lim_{\delta \to 0} \limsup_{\varepsilon \to 0} \int_0^T \int_{\Td} |f_\eps \ast \varphi_{\delta}(t,x) - f_\eps(t,x)|^2 \diff x \diff t = 0.
\end{equation}
uniformly for all $\eps$.
\end{itemize}
Then, $\{f_\eps\}$ is compact in time in $L^2((0,T)\times \Td)$, i.e.
\begin{equation}\label{item2compactness}
\lim_{h \to 0} \limsup_{\varepsilon \to 0} \int_0^{T-h} \int_{\Td} |f_\eps(t+h,x) - f_\eps(t,x)|^2 \diff x \diff t \to 0 \mbox{ as } h \to 0
\end{equation}
and so, it is compact in $L^2((0,T)\times \Td)$.
\end{lem}

 We recall that
\begin{equation*}
 \|\nabla^{k}\varphi_{\delta}\|_{L^{p}(\Td)}\le\frac{C}{\delta^{k+d-d/p}},
\end{equation*}
and for any function $g\in L^{p}(\Td)$,
\begin{equation*}
\|g\ast\varphi_{\delta}\|_{L^{p}(\Td)}\le \|\varphi_{\delta}\|_{L^{p}(\Td)}\|g\|_{L^{1}(\Td)}.	
\end{equation*}

\begin{proof}
Using the mollifiers with $\delta = \delta(h)$ depending on $h$ to be specified later in the way that $\delta(h) \to 0$ as $h\to 0$, we first split
\begin{align*}
\int_{0}^{T-h}\int_{\Td}|f_{\eps}(t+h,x)-f_{\eps}(t,x)&|^2\diff x \diff t\le 4\int_{0}^{T-h}\int_{\Td}|f_{\eps}(t,x)-f_{\eps}(t,\cdot)\ast\varphi_{\delta}(x)|^{2}\diff x \diff t\\
&+4\int_{0}^{T-h}\int_{\Td}|f_{\eps}(t+h,x)-f_{\eps}(t+h,\cdot)\ast\varphi_{\delta}(x)|^{2}\diff x \diff t\\
&+4\int_{0}^{T-h}\int_{\Td}|f_{\eps}(t+h,\cdot)\ast\varphi_{\delta}(x)-f_{\eps}(t,\cdot)\ast\varphi_{\delta}(x)|^{2}\diff x \diff t.
\end{align*}
When we apply limit $\lim_{h\to 0} \limsup_{\varepsilon\to0}$, the first and second term vanish due to \eqref{item2compactness_2}.
It remains to study the third term which reads
\begin{align*}
\int_{0}^{T-h}&\int_{\Td}|f_{\eps}(t+h,\cdot)\ast\varphi_{\delta}(x)-f_{\eps}(t,\cdot)\ast\varphi_{\delta}(x)|^{2}\diff x \diff t= \\
&=\int_{0}^{T-h}\int_{\Td}\left|\int_{t}^{t+h}\p_{t}f_{\eps}(s,\cdot)\ast\varphi_{\delta}(x) \diff s\right|^{2}\diff x \diff t =\int_{0}^{T-h}\int_{\Td}\left|\int_{t}^{t+h}J\ast \nabla^k\varphi_{\delta}(s,x) \diff s \right|^{2}\diff x \diff t\\
&\le C h \int_{0}^{T-h}\int_{\Td}\int_{t}^{t+h}\left|J\ast \nabla^k\varphi_{\delta}(s,x)\right|^{2} \diff s \diff x \diff t,
\end{align*}
where we used Jensen's inequality. We perform the change of variables $s \mapsto v=\f{s-t}{h}$, use Fubini's theorem, and obtain  
\begin{equation*}
 h\int_{0}^{T-h}\int_{\Td}\int_{t}^{t+h}\left|J\ast\nabla^k\varphi_{\delta}(s,x)\right|^{2} \diff s\diff x \diff t=h^{2}    \int_{0}^{1}\int_{\Td}\int_{0}^{T-h}\left|J \ast \nabla^k\varphi_{\delta}(vh+t,x)\right|^{2}\diff t\diff x \diff v.
\end{equation*}
Then we use the change of variables $t \mapsto \tau=v\,h+t$ and obtain 
\begin{multline*}
 h^{2}\int_{0}^{1}\int_{\Td}\int_{0}^{T-h}\left|J \ast \nabla^k\varphi_{\delta}(vh+t,x)\right|^{2}\diff t\diff x \diff v=  \\ = h^{2}\int_{0}^{1}\int_{\Td}\int_{vh}^{T+h(v-1)}\left|J_{i}\ast \nabla^k\varphi_{\delta}(\tau,x)\right|^{2}d\tau \diff x \diff v \le \f{h^{2}}{\delta^{2k+d}}\|J_{\eps}\|^2_{L^{1}_{t,x}}.
\end{multline*}
Using the $L^{1}((0,T)\times\Td)$ bound on $\{J_{\eps}\}$ and choosing $\delta$ such that $\delta^{2k+d}=h$ we conclude that
\begin{equation*}
   \lim_{h\to 0} \limsup_{\varepsilon\to0}  \int_{0}^{T-h}\int_{\Td}|f_{\eps}(t+h,x)-f_{\eps}(t,x)|^2\diff x \diff t\le \theta(h).
\end{equation*}

Combined with the compactness in space~\eqref{item2compactness} and the Fréchet-Kolmogorov theorem we obtain the compactness of $\{f_{\eps}\}$ in $L^{2}((0,T)\times \Td)$. 
\end{proof}

\begin{rem}
Compared with the usual version of the Fréchet-Kolmogorov theorem, one would expect that the condition for compactness in space should read
\begin{equation}\label{eq:compactnes_space_usual}
\lim_{y \to 0} \limsup_{\varepsilon \to 0} \int_0^T \int_{\Td} |f_\eps(t,x+y) - f_\eps(t,x)|^2 \diff x \diff t = 0.
\end{equation}
However, by a careful inspection of the proof, \eqref{item2compactness_2} is sufficient and in fact, in the proof one deduces \eqref{item2compactness_2} from \eqref{eq:compactnes_space_usual}. 
\end{rem}

\bibliographystyle{abbrv}
\bibliography{NEWNEW_new_file}

\begin{thebibliography}{10}

\bibitem{MR1612250}
G.~Alberti and G.~Bellettini.
\newblock A nonlocal anisotropic model for phase transitions. {I}. {T}he
  optimal profile problem.
\newblock {\em Math. Ann.}, 310(3):527--560, 1998.

\bibitem{bernoff2016biological}
A.~J. Bernoff and C.~M. Topaz.
\newblock Biological aggregation driven by social and environmental factors: a
  nonlocal model and its degenerate {C}ahn-{H}illiard approximation.
\newblock {\em SIAM J. Appl. Dyn. Syst.}, 15(3):1528--1562, 2016.

\bibitem{bourgain2001another}
J.~Bourgain, H.~Brezis, and P.~Mironescu.
\newblock Another look at {S}obolev spaces.
\newblock In {\em Optimal control and partial differential equations}, pages
  439--455. IOS, Amsterdam, 2001.

\bibitem{byrne_modelling_2004}
H.~{Byrne} and L.~{Preziosi}.
\newblock Modelling solid tumour growth using the theory of mixtures.
\newblock {\em Math. Med. Biol.}, 20(4):341--366, 2003.

\bibitem{carrillo2019aggregation}
J.~A. Carrillo, K.~Craig, and Y.~Yao.
\newblock Aggregation-diffusion equations: dynamics, asymptotics, and singular
  limits.
\newblock In {\em Active particles. {V}ol. 2. {A}dvances in theory, models, and
  applications}, Model. Simul. Sci. Eng. Technol., pages 65--108.
  Birkh\"{a}user/Springer, Cham, 2019.

\bibitem{MR4221297}
L.~Cherfils, H.~Fakih, M.~Grasselli, and A.~Miranville.
\newblock A convergent convex splitting scheme for a nonlocal
  {C}ahn-{H}illiard-{O}ono type equation with a transport term.
\newblock {\em ESAIM Math. Model. Numer. Anal.}, 55(suppl.):S225--S250, 2021.

\bibitem{david2021incompressible}
N.~David and M.~Schmidtchen.
\newblock On the incompressible limit for a tumour growth model incorporating
  convective effects.
\newblock {\em arXiv preprint arXiv:2103.02564, to appear in Comm. Pure Appl.
  Math.}, 2021.

\bibitem{MR4093616}
E.~Davoli, H.~Ranetbauer, L.~Scarpa, and L.~Trussardi.
\newblock Degenerate nonlocal {C}ahn-{H}illiard equations: well-posedness,
  regularity and local asymptotics.
\newblock {\em Ann. Inst. H. Poincar\'{e} C Anal. Non Lin\'{e}aire},
  37(3):627--651, 2020.

\bibitem{MR4248454}
E.~Davoli, L.~Scarpa, and L.~Trussardi.
\newblock Local asymptotics for nonlocal convective {C}ahn-{H}illiard equations
  with {$W^{1,1}$} kernel and singular potential.
\newblock {\em J. Differential Equations}, 289:35--58, 2021.

\bibitem{MR4198717}
E.~Davoli, L.~Scarpa, and L.~Trussardi.
\newblock Nonlocal-to-local convergence of {C}ahn-{H}illiard equations:
  {N}eumann boundary conditions and viscosity terms.
\newblock {\em Arch. Ration. Mech. Anal.}, 239(1):117--149, 2021.

\bibitem{degond2022multi}
P.~Degond, S.~Hecht, M.~Romanos, and A.~Trescases.
\newblock Multi-species viscous models for tissue growth: incompressible limit
  and qualitative behaviour.
\newblock {\em J. Math. Biol.}, 85:16, 2022.

\bibitem{delgadino2018convergence}
M.~G. Delgadino.
\newblock Convergence of a one-dimensional {C}ahn-{H}illiard equation with
  degenerate mobility.
\newblock {\em SIAM J. Math. Anal.}, 50(4):4457--4482, 2018.

\bibitem{DEBIEC2021204}
T.~Dębiec, B.~Perthame, M.~Schmidtchen, and N.~Vauchelet.
\newblock Incompressible limit for a two-species model with coupling through
  brinkman's law in any dimension.
\newblock {\em J. Math. Pures Appl.}, 145:204--239, 2021.

\bibitem{elbar-mason-perthame-skrzeczkowski}
C.~Elbar, M.~Mason, B.~Perthame, and J.~Skrzeczkowski.
\newblock From {V}lasov equation to degenerate nonlocal {C}ahn-{H}illiard
  equation.
\newblock {\em arXiv preprint arXiv:2208.01026}, 2022.

\bibitem{elbar2021degenerate}
C.~Elbar, B.~Perthame, and A.~Poulain.
\newblock Degenerate {C}ahn-{H}illiard and incompressible limit of a
  {K}eller-{S}egel model.
\newblock {\em arXiv preprint arXiv:2112.10394, to appear in Commun. Math.
  Sci.}, 2021.

\bibitem{MR1377481}
C.~M. Elliott and H.~Garcke.
\newblock On the {C}ahn-{H}illiard equation with degenerate mobility.
\newblock {\em SIAM J. Math. Anal.}, 27(2):404--423, 1996.

\bibitem{falco2022local}
C.~Falc{\'o}, R.~E. Baker, and J.~A. Carrillo.
\newblock A local continuum model of cell-cell adhesion.
\newblock {\em arXiv preprint arXiv:2206.14461}, 2022.

\bibitem{MR4241616}
S.~Frigeri, C.~G. Gal, and M.~Grasselli.
\newblock Regularity results for the nonlocal {C}ahn-{H}illiard equation with
  singular potential and degenerate mobility.
\newblock {\em J. Differential Equations}, 287:295--328, 2021.

\bibitem{MR3688414}
C.~G. Gal, A.~Giorgini, and M.~Grasselli.
\newblock The nonlocal {C}ahn-{H}illiard equation with singular potential:
  well-posedness, regularity and strict separation property.
\newblock {\em J. Differential Equations}, 263(9):5253--5297, 2017.

\bibitem{MR3072989}
C.~G. Gal and M.~Grasselli.
\newblock Longtime behavior of nonlocal {C}ahn-{H}illiard equations.
\newblock {\em Discrete Contin. Dyn. Syst.}, 34(1):145--179, 2014.

\bibitem{MR1453735}
G.~Giacomin and J.~L. Lebowitz.
\newblock Phase segregation dynamics in particle systems with long range
  interactions. {I}. {M}acroscopic limits.
\newblock {\em J. Statist. Phys.}, 87(1-2):37--61, 1997.

\bibitem{MR1638739}
G.~Giacomin and J.~L. Lebowitz.
\newblock Phase segregation dynamics in particle systems with long range
  interactions. {II}. {I}nterface motion.
\newblock {\em SIAM J. Appl. Math.}, 58(6):1707--1729, 1998.

\bibitem{KNOPF2021236}
P.~Knopf and A.~Signori.
\newblock On the nonlocal {C}ahn–{H}illiard equation with nonlocal dynamic
  boundary condition and boundary penalization.
\newblock {\em J. Differential Equations}, 280(4):236--291, 2021.

\bibitem{ladyzhenskaya1968linear}
O.~Ladyzhenskaya, V.~Solonnikov, and N.~Ural’tseva.
\newblock Linear and quasilinear equations of parabolic type, transl. math.
\newblock {\em Monographs, Amer. Math. Soc}, 23, 1968.

\bibitem{MR1465184}
G.~M. Lieberman.
\newblock {\em Second order parabolic differential equations}.
\newblock World Scientific Publishing Co., Inc., River Edge, NJ, 1996.

\bibitem{LISINI2012814}
S.~Lisini, D.~Matthes, and G.~Savaré.
\newblock Cahn–{H}illiard and thin film equations with nonlinear mobility as
  gradient flows in weighted-{W}asserstein metrics.
\newblock {\em J. Differential Equations}, 253(2):814--850, 2012.

\bibitem{MR2581977}
D.~Matthes, R.~J. McCann, and G.~Savar\'{e}.
\newblock A family of nonlinear fourth order equations of gradient flow type.
\newblock {\em Comm. Partial Differential Equations}, 34(10-12):1352--1397,
  2009.

\bibitem{MR4408204}
S.~Melchionna, H.~Ranetbauer, L.~Scarpa, and L.~Trussardi.
\newblock From nonlocal to local {C}ahn-{H}illiard equation.
\newblock {\em Adv. Math. Sci. Appl.}, 28(2):197--211, 2019.

\bibitem{MR3362777}
S.~Melchionna and E.~Rocca.
\newblock On a nonlocal {C}ahn-{H}illiard equation with a reaction term.
\newblock {\em Adv. Math. Sci. Appl.}, 24(2):461--497, 2014.

\bibitem{MR4001523}
A.~Miranville.
\newblock {\em The {C}ahn-{H}illiard equation. Recent advances and
  applications}, volume~95 of {\em CBMS-NSF Regional Conference Series in
  Applied Mathematics}.
\newblock Society for Industrial and Applied Mathematics (SIAM), Philadelphia,
  PA, 2019.
\newblock Recent advances and applications.

\bibitem{Perthame-Hele-Shaw}
B.~Perthame, F.~Quir\'{o}s, and J.~L. V\'{a}zquez.
\newblock The {H}ele-{S}haw asymptotics for mechanical models of tumor growth.
\newblock {\em Arch. Ration. Mech. Anal.}, 212(1):93--127, 2014.

\bibitem{MR2041005}
A.~C. Ponce.
\newblock An estimate in the spirit of {P}oincar\'{e}'s inequality.
\newblock {\em J. Eur. Math. Soc. (JEMS)}, 6(1):1--15, 2004.

\bibitem{MR4365199}
E.~Rocca, L.~Scarpa, and A.~Signori.
\newblock Parameter identification for nonlocal phase field models for tumor
  growth via optimal control and asymptotic analysis.
\newblock {\em Math. Models Methods Appl. Sci.}, 31(13):2643--2694, 2021.

\bibitem{MR3014456}
T.~Roub\'{\i}\v{c}ek.
\newblock {\em Nonlinear partial differential equations with applications},
  volume 153 of {\em International Series of Numerical Mathematics}.
\newblock Birkh\"{a}user/Springer Basel AG, Basel, second edition, 2013.

\bibitem{MR2082242}
E.~Sandier and S.~Serfaty.
\newblock Gamma-convergence of gradient flows with applications to
  {G}inzburg-{L}andau.
\newblock {\em Comm. Pure Appl. Math.}, 57(12):1627--1672, 2004.

\bibitem{MR2836361}
S.~Serfaty.
\newblock Gamma-convergence of gradient flows on {H}ilbert and metric spaces
  and applications.
\newblock {\em Discrete Contin. Dyn. Syst.}, 31(4):1427--1451, 2011.

\bibitem{takata2018simple}
S.~Takata and T.~Noguchi.
\newblock A simple kinetic model for the phase transition of the van der
  {W}aals fluid.
\newblock {\em Journal of Statistical Physics}, 172(3):880--903, 2018.

\bibitem{yan2012extension}
M.~Yan.
\newblock Extension of convex function.
\newblock {\em arXiv preprint arXiv:1207.0944}, 2012.

\end{thebibliography}
\end{document}